\documentclass[11pt]{amsart}
\usepackage{amssymb,amsmath,amsthm}
\usepackage{graphicx}%,psfrag,color,pst-grad}
\numberwithin{equation}{section}

\newtheorem{Thm}{Theorem}[section]
\newtheorem{Lem}{Lemma}[section]
\newtheorem{Prop}{Proposition}[section]

\newtheorem{Rmk}{Remark}[section]

\theoremstyle{definition}

\title{Averaging of time--periodic systems
without a small parameter}

\title[Averaging without a small parameter]
{Averaging of time--periodic systems without\\ a small parameter}
\author[M.D. Chekroun, M. Ghil, J.Roux and F. Varadi]{}

\email{chekro@lmd.ens.fr {\rm (M.D. Chekroun)}}

\thanks{The present manuscript has been published as:
\textsc{M.D. Chekroun, M. Ghil, J.Roux and F. Varadi}, in {\it
Discr. and Cont. Dyn. Syst., Ser. A}, {\bf 14}(4), (2006) 753--782,
and includes here furthermore, in Appendix 2, the proof of Lemma 6.1
of this published version.}

\begin{document}
\maketitle

 \centerline{\scshape Micka\"el D. Chekroun}
{\footnotesize
 \centerline{Laboratoire Jacques-Louis Lions}
 \centerline{Universit\'e Pierre et Marie Curie, Paris, France}
  \centerline{and}
   \centerline{Environmental Research and Teaching Institute}
   \centerline{\'{E}cole Normale Sup\'erieure, Paris, France}
   }

\medskip
\centerline{\scshape Michael Ghil}

{\footnotesize

\centerline{D\'{e}partement Terre-Atmosph\`{e}re-Oc\'{e}an and}
\centerline{Laboratoire de M\'{e}t\'{e}orologie Dynamique du
CNRS/IPSL} \centerline{\'{E}cole Normale Sup\'{e}rieure, Paris,
France} \centerline{and} \centerline{Department of Atmospheric and
Oceanic Sciences and} \centerline{Institute of Geophysics and
Planetary Physics} \centerline{University of California, Los
Angeles, U.S.A.} }

\medskip

 \centerline{\scshape Jean Roux}

{\footnotesize \centerline{Laboratoire de M\'{e}t\'{e}orologie
Dynamique du CNRS/IPSL}
%\centerline{and}
% \centerline{Environmental Research and Teaching Institute}
 \centerline{\'{E}cole Normale Sup\'erieure, Paris, France}
   }

 \medskip
 \centerline{\scshape Ferenc Varadi}

{\footnotesize \centerline{Institute of Geophysics and Planetary
Physics} \centerline{University of California, Los Angeles, U.S.A.}
}

\medskip

%\begin{quote}{\normalfont\fontsize{8}{10}\selectfont
\begin{abstract}
In this article, we present a new approach to averaging in
non-Hamiltonian systems with periodic forcing. The results here do
not depend on the existence of a small parameter. In fact, we show
that our averaging method fits into an appropriate nonlinear
equivalence problem, and that this problem can be solved formally by
using the Lie transform framework to linearize it. According to this
approach, we derive formal coordinate transformations associated
with both first-order and higher-order averaging, which result in
more manageable formulae than the classical ones.

Using these transformations, it is possible to correct the solution
of an averaged system by recovering the oscillatory components of
the original non-averaged system. In this framework, the inverse
transformations are also defined explicitly by formal series; they
allow the estimation of appropriate initial data for each
higher-order averaged system, respecting the equivalence relation.

Finally, we show how these methods can be used for identifying and
computing periodic solutions for a very large class of nonlinear
systems with time-periodic forcing. We test the validity of our
approach by analyzing both the first-order and the second-order
averaged system for a problem in atmospheric chemistry.
%\par}
%\end{quote}
\end{abstract}

\numberwithin{equation}{section}

\section{Introduction}

The theory of averaging falls within the more general theory of
normal forms; it arises from perturbation analysis and is thus
formulated usually for equations which contain a small parameter
$\epsilon$. A vast literature treats this case (e.g. \cite{a, ch,
esd, gh, hale, lo, mur, per, sv}, and references therein). In this
paper we develop a time-averaging theory for the case of
time-periodic nonautonomous systems without an explicit dependence
on a small parameter. The theory is presented here formally, subject
to a conditional statement to be explained later. It is shown to
apply numerically to our test problem, which does contain a
``large-amplitude" periodic perturbation.

Among the methods of derivation of averaged systems, those based on
Lie transforms offer an efficient computational framework (cf.
\cite{ch, na}, or more recently \cite{jap}). The rigorous use of
such techniques, however,  is confined so far to the construction of
near--identity transformations expanded in a small parameter.

In order to overcome the main issues raised by the lack of
$\epsilon$-parametrization, we consider the problem in terms of
``differentiable equivalence" between the time-dependent vector
field and a corresponding averaged form, allowing us to put the
problem in an appropriate Lie transform framework.

In Hamiltonian dynamics, G. Hori \cite{h} was the first to use Lie
series for coordinate transformations, while A. Deprit \cite{d}
introduced the slightly different formalism of Lie transforms; the
two were shown to be equivalent by J. Henrard and J. Roels
\cite{her}. They developed a new, efficient way to construct
canonical transformations for perturbation theory in celestial
mechanics. The Poisson bracket in their formulae can be replaced by
the Lie bracket to obtain a representation of general, non-canonical
transformations \cite{ch, he, k1}. Nowadays, Lie transforms are
applied in many research fields, such as artificial satellites
\cite{ste}, particle accelerator design \cite{dra} and optical
devices \cite{gab}, to mention just a few.

The object of the present article is to perform averaging analysis
for a $T$-periodic $N$-dimensional system ($N\geq 2, T>0$) of ODEs
\begin{equation}\label{generalsystem}
 \frac{dx}{dt}=Y(t,x)=Y_{t}(x),\quad x\in \Omega,
\end{equation}
where $\Omega$ is an open subset of $\mathbb{R}^{N}$. The precise
framework is given in \S \ref{LTS}.

In the present paper, we consider the averaging process via an
appropriate notion of differentiable equivalence (Definition
\ref{Equivalence}). This equivalence concept allows us to determine
the  underlying time-dependent family of diffeomorphisms, as the
solution of the nonlinear functional equation (\ref{bigproblem}).
This equation plays a central role in our approach, as it allows one
to generalize the Lie transform formalism. We put it into a pullback
form that falls into this formalism via suspension (Proposition
\ref{suspensiontool} and Appendix 1). This process allows us to
obtain a computable formal solution of equation (\ref{bigproblem})
(Theorem \ref{diffeorepre}), based on the solvability of a recursive
family of linear PDEs, typically known as Lie's equations \cite{ch,
hen2, v}. It is shown in \S \ref{SolvingLie} and \S \ref{corectave}
that Lie's equations can be solved easily in our framework.

This approach is consistent with other classical equivalence or
conjugacy problems, in  which the so-called homological equations
linearizing
 such problems are a cornerstone in proving Anosov's theorem or
 Hartman--Grobman theorem (cf. \cite{a} for details). A recent proof of
 Kolmogorov's theorem on the conservation of invariant tori also
 used this classical framework \cite{hu}. The essential difference
 between our approach and the classical one is that
  (\ref{bigproblem}) involves a ``conjugacy" (see Definition \ref{Equivalence}
  for the right
  notion) between a nonautonomous
 field in the original problem and an autonomous one in the
 transformed problem. The formal developments in \S \ref{LTS} to
 \S \ref{seccondinit} are contingent upon Lemma 6.1.
 This Lemma will be proven in a subsequent paper (in preparation), and permits a
 rigorous proof of the existence of a solution to  equation
 (\ref{bigproblem}). The numerical results in \S \ref{numtest} and
 \S \ref{period} demonstrate the usefulness of the present approach and the
 plausibility of the forthcoming rigorous results.

Our approach provides, furthermore, simpler formulae for
higher-order averaged systems, even in the classical context of
perturbations scaled by a small parameter $\epsilon$. This is shown
in Proposition \ref{nthaveraging} and is due to the suspension
associated with solving (\ref{bigproblem}); see \S \ref{corectave},
\S \ref{heuristic} and \cite{esd, per}.

We show that the classical change of coordinates on the initial data
respects the differentiable equivalence relation (\ref{bigproblem})
at time $t=0$. It arises here in a more natural way from the
generator of the family of inverse diffeomorphims at time $t=0$; see
Proposition \ref{inverserepre} and \cite{hen}.

We demonstrate that the Lie transform framework permits gaining CPU
time with respect to the standard numerical methods by, loosely
speaking, integrating the averaged system with a larger step size;
this is followed by performing efficient corrections with a step
size adapted to the oscillations of the forcing terms. We test the
validity of this numerical approach by computing the second-order
terms in our atmospheric-chemistry problem. More precisely, we build
a second-order averaged system in this test problem and the relevant
correction in the same framework, after applying the
averaging-correction method at the first order.

In \S \ref{LTS} we describe how  Lie transforms allow us to resolve
formally a differentiable equivalence problem between a
time-periodic system and an autonomous one. In \S \ref{higherorder},
we give a formal algorithm of computing higher-order averaged
systems based on Lie transforms and related corrections; furthermore
we compare our results to the classical ones. In \S
\ref{seccondinit} we clarify how to compute initial data which
satisfy the differentiable equivalence relation. Numerical results
and advantages of this technique are discussed in \S \ref{numtest}
for a simple model of atmospheric chemistry. In \S \ref{period}, we
comment on the use of these methods for constructing periodic
orbits, in nonautonomous {\em dissipative} systems with periodic
forcing. In the Appendix 1, the Lie transform framework is presented
as a tool for solving pullback problems in general.

\section{The Lie transform setting and the differentiable equivalence
problem}\label{LTS}
\subsection{Fitting the differentiable equivalence problem into the Lie
transform formalism}\label{fiteq} For all the rest of the paper
$\Omega$ will be an open subset of $\mathbb{R}^{N}$, $N$ will be a
positive integer, and $T$ will be a positive real.

For subsequent computations we shall consider
$\mathcal{Q}^{k}(\Omega)$ as the class of continuous functions
$f:\mathbb{R}^{+}\times \Omega\rightarrow \mathbb{R}^{N}$ which are
$T$-periodic in $t$ for each $x\in \Omega$ and which have
$k$-continuous partial derivatives with respect to $x$ for $x\in
\Omega$ and $t$ in $\mathbb{R}^{+}$. More precisely we will work
with $\mathcal{Q}^{\infty}(\Omega)=\bigcap_{k\geq
0}\mathcal{Q}^{k}(\Omega)$ when we deal with Lie transforms as often
it is requires in this framework for dependence on the spatial
variables (cf. Appendix 1).

There are known results on the boundedness and the global existence
in time of solutions of ODEs. For instance, by an application of
Gronwall's lemma it is easy to show that if there exists $\alpha,
\beta \in \mathcal{C}^{0}(\mathbb{R}^{+},\mathbb{R}^{+})\cap
L^{1}(\mathbb{R}^{+},\mathbb{R})$ such that $\Vert Y(t,x)\Vert\leq
\alpha(t)\Vert x \Vert+\beta(t)$, for all
$(t,x)\in\mathbb{R}^{+}\times\Omega$,
 then every solution of $\dot{x}=Y(t,x)$ is
bounded. Thus, using the continuation theorem (cf. theorem 2.1 of
\cite{hale} for instance), we can prove that the solutions of the
above equations are global in time.

Taking into consideration this fact, we make the following
assumption.
\begin{itemize}
\item[($\lambda$)] All solutions remain in $\Omega$, for all
considered non-autonomous and autonomous vector fields defined on
$\Omega$.
\end{itemize}

Henceforth, for all $1\leq k\leq\infty$, we introduce
$\mathcal{P}^{k}(\Omega):=\mathcal{Q}^{k}(\Omega)\cap
\{f:\mathbb{R}^{+}\times\Omega\rightarrow\Omega\}$ and
$\mathcal{C}^{k}(\Omega):=\mathcal{C}^{k}(\Omega,\Omega)$.

 We describe some concepts related to the subject of this
paper, dealing with differentiable equivalence between a {\em
nonautonomous} vector field  and an {\em autonomous} one.

\vspace{1ex} \noindent {\bf Definition 2.1. }\label{Equivalence}
{\em Let $r$ be a positive integer, let $Y \in
\mathcal{P}^{r}(\Omega)$ and $Z \in \mathcal{C}^{r}(\Omega)$, then
$Y$ and $Z$ are said to be $\mathcal{P}^{k}_{diff}$-equivalent
($k\leq r$), if there exists a map $\Phi \in
\mathcal{P}^{k}(\Omega)$ such that for all $t \in \mathbb{R}^{+}$,
$\Phi_t :=\Phi(t,\cdot):\Omega \rightarrow \Omega$ is a
$\mathcal{C}^{k}$-diffeomorphism, which carries the solutions
$x(t,x_0)$ of $\dot{x}=Y(t,x)$, for $x(0,x_0)=x_0$ varying in
$\Omega$, into the solutions $z(t,z_0)$ of $\dot{z}=Z(z)$, in the
sense that:
\begin{equation}\label{equiv-sol}
x(t,x_0)=\Phi_t(z(t,\Phi_{0}^{-1}(x_0))), \mbox{ for all } t \in
\mathbb{R}^{+}, (z_0=\Phi_{0}^{-1}(x_0)).
\end{equation}}
%\refstepcounter{theorem}

Using the notations in this definition, we shall say that if
(\ref{equiv-sol}) is satisfied for a pair $(x, z)$ of solutions,
then $x$ and $z$ are in {\em $\mathcal{P}^{k}_{diff}$-correspondence
by $\Phi$}, or, sometimes that they are in {\em
$\mathcal{P}^{k}_{diff}$-correspondence by
$(\Phi_t)_{t\in\mathbb{R}^+}$}.

Furthermore, we will often use for the underlying transformations
the functional space
\begin{eqnarray}
\mathcal{P}^{k}_{d}(\Omega):=\{\Phi \in \mathcal{P}^{k}(\Omega)
\mbox{ such that, for all } t \in \mathbb{R}^{+}, \Phi_t :\Omega
\rightarrow \Omega \nonumber \\ \mbox{ is a }
\mathcal{C}^{k}\mbox{-diffeomorphism} \mbox{, and } t\rightarrow
\Phi_t^{-1} \mbox{ is } \mathcal{C}^{1}\}, \end{eqnarray} where the
smoothness assumption on the map $t\rightarrow \Phi_t^{-1}$, will be
apparent from the proof of Lemma \ref{conjugpull}. \vspace{1ex}

\begin{Rmk} It is important to note here that the
$\mathcal{P}^{k}_{diff}$-equivalence of Definition \ref{Equivalence}
does not conserve the period of the trajectories although it
preserves parametrization by time and sense, and, in this meaning,
realizes a compromise between the classical notions of conjugacy and
equivalence (e.g. \cite{gh}). This will be essential in \S
\ref{period}.
\end{Rmk}

Let us recall the classical definition of pullback of autonomous
vector fields.

\vspace{1ex}
 \noindent{\bf Definition 2.2.}\label{definitionimage} {\em Let $U$ and
$\mathcal{O}$ be open subsets of $\mathbb{R}^{N}$. Let $X$ be a
vector field of class at least $\mathcal{C}^{0}$ on $\mathcal{O}$,
and let $\Phi \in \mathcal{C}^{1}(U,\mathcal{O})$ be a
$\mathcal{C}^{1}$-diffeomorphism. The pullback of X on $\mathcal{O}$
by $\Phi$ is the vector field defined on $U$ as the map
\begin{equation}
y\to (\Phi^{-1}\ast X)(y):=(D\Phi ^{-1})(\Phi (y)) \cdot X(\Phi
(y)),\mbox{ for all } y\in U=\Phi^{-1}(\mathcal{O}).
\end{equation}}
%\refstepcounter{theorem}

\begin{Rmk} Sometimes we will employ the notation
$({\Phi}^{\ast}X)(y):=(D\Phi(y))^{-1}\cdot X(\Phi(y))$ for
$(\Phi^{-1}\ast X)(y)$.
\end{Rmk}

 The problem of $\mathcal{P}^{k}_{diff}$-equivalence  takes the form of a family of
``pulled back" problems via the
\begin{Lem}
\label{conjugpull} Let $Y$ and $Z$ as in Definition
\ref{Equivalence}. Let $x:\mathbb{R}^{+}\rightarrow \Omega$ a
solution of $\dot{x}=Y(t,x):=Y_t(x)\,,x(0)=x_0 \in \Omega$, and $z$
an integral curve of $Z$. Then, $x$ and $z$ are in
$\mathcal{P}^{k}_{diff}$-correspondence by $\Phi \in
\mathcal{P}^{k}_{d}(\Omega)$ if and only if $z$ is solution of the
IVP: $\dot{y}=(\Phi_{t}^{-1}\ast
Y_t)(y)+\partial_t\Phi_t^{-1}(\Phi_t(y))\,,\,y(0)=\Phi_0^{-1}(x_0)$.
\end{Lem}

\noindent
 {\em Proof:} The proof of this lemma is an obvious
application of the chain rule formula and Definition
\ref{definitionimage}.$\qquad \square$\vspace{2ex}

 Let us now introduce the projection
\begin{equation}\label{defpi}
\pi : \left\{
\begin{array}{l}
\mathbb{R}^{N+1} \to \mathbb{R}^{N}\\
(t,x) \to x
\end{array} \right..
\end{equation}

We define also a map, defined for each $t\in\mathbb{R}^{+}$, as
\begin{equation}\label{defit}
\mathcal{I}_t : \left\{
\begin{array}{l}
\mathbb{R}^{N} \to \mathbb{R}^{N+1}\\
x \to (t,x)
\end{array} \right.,
\end{equation}
which will be $T$-periodic in time $t$.

According to Lemma \ref{conjugpull}, if there exists a map $\Phi \in
\mathcal{P}^{k}_{d}(\Omega)$ which satisfies the following nonlinear
functional equation, (called in this work {\em equivalence problem})
\begin{equation}\label{bigproblem}
\Phi_{t}^{-1}\ast Y_t+\partial_t\Phi^{-1}_t\circ\Phi_t=Z, \mbox{ for
all } t\in\mathbb{R}^{+},
\end{equation}
then every integral curve of $Z$ is in
$\mathcal{P}^{k}_{diff}$-correspondence by $\Phi$ to a unique
solution of: $\dot{x}=Y(t,x)$, that is the vector fields $Y$ and $Z$
are $\mathcal{P}^{k}_{diff}$-equivalent.

The problem is then to solve in
$\Phi\in\mathcal{P}^{k}_{d}(\Omega)$, the equation
(\ref{bigproblem}) for a given pair of fixed vector fields
$Y\in\mathcal{P}^{r}(\Omega)$ (non-autonomous) and
$Z\in\mathcal{C}^{r}(\Omega)$ (autonomous).

In order to show how Lie transforms formalism (e.g., \cite{ch, d,
her} and Appendix 1) can be used for solving  formally equation
(\ref{bigproblem}) with $r=k=\infty$, we put this equation in an
appropriate pullback form by the Proposition \ref{suspensiontool}
below.

For that we introduce some notations. We denote by $\widetilde{Y}$
the so-called $t$-suspended vector field associated with
(\ref{generalsystem}) written in terms of the enlarged vector of
dynamical variables, $\widehat{x}=(t,x_{1},\cdots,x_{N})^T$, thus if
$x$ is a solution of $\dot{x}=Y(t,x)$, then $\widehat{x}$ satisfies
\begin{equation}  \label{suspendedY}
\frac{d\widehat{x}}{dt}=\widetilde{Y}(\widehat{x})
 \mbox{  where  }\widetilde{Y}=\left[
1_{\mathbb{R}},Y \right]^{T}.
\end{equation}
Note that by assumption ($\lambda$),
 every solution $\widehat{x}$ of
(\ref{suspendedY}) is contained in $\widetilde{\Omega}:=\{(t,x)\in
\mathbb{R}^{+}\times\Omega\}$, for the rest of the paper.

\begin{Rmk} Note that for the rest of the paper the $t$-suspended form of
a (time-dependent or not) vector field $X$ on $\mathbb{R}^N$ will be
defined as $\widetilde{X}=\left[ 1_{\mathbb{R}},X \right]^{T}$, and
the {\em flat form} of $X$ will be defined as $\left[
0_{\mathbb{R}}, X\right]^{T}$. The same notation $\widetilde{X}$
will be used for the latter, by abusing the notation. No confusion
will be made with regard to the context. \end{Rmk}

\begin{Prop}\label{suspensiontool}
Let $Z\in\mathcal{C}^{\infty}(\Omega)$ be an autonomous vector field
and let $Y \in \mathcal{P}^{\infty}(\Omega)$, with $\widetilde{Z}$
and $\widetilde{Y}$ as their respective $t$-suspended forms. Let
$\Theta \in \mathcal{P}^{\infty}_{d}(\widetilde{\Omega})$. If
$\Theta $ satisfies $\Theta^{-1}_t\ast\widetilde{Y}=\widetilde{Z}$,
for all $t \geq 0$, then the family of
$\mathcal{C}^{\infty}$-diffeomorphisms
$(\pi\circ\Theta_t\circ\mathcal{I}_t)_{t\in\mathbb{R}^{+}}$, acting
on $\Omega$, satisfies the equation (\ref{bigproblem}). The converse
is also true.
\end{Prop}

\begin{proof}
The proof relies on basic calculus and the definitions listed
above.
\end{proof}
%$\qquad \square$\vspace{2ex}

We illustrate now the strategy used to solve equation
(\ref{bigproblem}) for
$\widetilde{Z}=\widetilde{\overline{Y}}:=\left[1,\overline{Y}
\right]^{T}$ with $\overline{Y}$ being the classical first averaged
system, namely $\frac{1}{T}\int_0^{T} Y ds$. The same reasoning will
be used for higher-order averaged systems in \S \ref{higherorder}.
By Proposition \ref{suspensiontool}, the problem of finding a
solution of (\ref{bigproblem}) falls into the Lie transform
formalism if $\widetilde{Y}$ and $\widetilde{\overline{Y}}$ are
given by $\tau$-series (cf. Appendix 1 and Theorem 7.1). This point
of view, according to the solvability of Lie's equations, will then
allow us to solve formally (\ref{bigproblem}) in \S
\ref{solbigprob}.

Starting from that point, denote $\widetilde{\overline{Y}}$ by
$\widetilde{Y}_{ave}$. In order to write $\widetilde{Y}_{ave}$ and
$\widetilde{Y}$ as $\tau$-series we introduce a natural auxiliary
real parameter $\tau$,  by the following expressions
\begin{align}  \label{f14b}
\widetilde{Y}_{ave,\tau}&:= \left[ 1_{\mathbb{R}},
0_{\mathbb{R}^{N}} \right]^{T} +\tau \cdot \left[ 0_{\mathbb{R}},
\overline{Y} \right]^{T},
\\  \label{f15b}
\widetilde{Y}_{\tau}&:= \left[ 1_{\mathbb{R}}, 0_{\mathbb{R}^{N}}
\right]^{T} +\tau \cdot \left[ 0_{\mathbb{R}}, Y
\right]^{T}=\widetilde{Y}_{0}^{(0)}+\tau \cdot
\widetilde{Y}_{1}^{(0)},
\end{align}
with the obvious definitions for the vector fields
$\widetilde{Y}_{0}^{(0)}$ and $\widetilde{Y}_{1}^{(0)}$, and
\begin{equation}\label{formuledonnees}
\widetilde{Y}_{k}^{(0)}=0_{\mathbb{R}^{N+1}}, \mbox{ for all
integers }\,k\geq2.
\end{equation}

Consider $\varphi_{\tau,t}$ the {\em semiflow at $\tau$} generated
by a $\tau$-suspended  smooth field $\widetilde{G}_t$ for each
non-negative $t$, {\it i.e.} $\varphi_{\tau,t}$ is the solution at
$\tau$ of
\begin{equation}  \label{f16}
S_{\widetilde{G}_t}\left\{
\begin{array}{c}
\frac{d\widehat{\xi}}{d\tau
}=\widetilde{G}_t(\widehat{\xi})=\widetilde{G}(t,\widehat{\xi}) \\
\end{array}
\right. ,\widehat{ \xi }=(\tau ,\xi )\in \mathbb{R}^{N+1},
\end{equation}
with a formal  expansion of $\widetilde{G}(t,\widehat{\xi})$ in
powers of $\tau$ given by
\begin{equation}\label{tauserieG}
\widetilde{G}(t,\widehat{\xi})=\left[ 1, G(t,\widehat{\xi})
\right]^{T}=\sum_{n\geq 0}\frac{\tau^n}{n!}\widetilde{G}_{n}(t,\xi
),
\end{equation}where
\begin{equation}\label{Gi}
\widetilde{G}_{0}(t,\xi)=\left[ 1, G_{0}(t,\xi) \right]^{T} \mbox{
and for } n\in \mathbb{Z}^{*}_{+} \,\mbox{,} \;
\widetilde{G}_{n}(t,\xi)=\left[ 0, G_{n}(t,\xi) \right]^{T},
\end{equation}
with   $G$ and $G_n$  smooth vector fields on $\mathbb{R}^{N}$.

\begin{Rmk} The generator system would be denoted by another notation, the
$\tau$-suspension and the $t$-suspension being not equivalent
generally. We slightly abuse the notation. \end{Rmk}

Applying the discussion of the Appendix 1 with $p=N+1$,
$A_{\tau}=\widetilde{Y}_{\tau}$ and
$B_{\tau}=\widetilde{Y}_{ave,\tau}$ for all $\tau$, and utilizing
the fact that  Lie's equations are independent of $\tau$, we obtain
that if there exists $\varphi_{\tau,t}$ such that
$(\varphi_{\tau,t})^{\ast
}\widetilde{Y}_{\tau}=\widetilde{Y}_{ave,\tau}$ for every
non-negative real $t$, then by Corollary 7.2 with
$H_{\tau,t}=G_{\tau,t}$, Lie's equations

\begin{equation}
\left\{
\begin{array}{l}
\widetilde{Y}_{0}^{(0)}=\left[ 1_{\mathbb{R}}, 0_{\mathbb{R}^{N}}
\right]^{T},  \\
\widetilde{Y}_{0}^{(1)}=\left[ 0_{\mathbb{R}}, \overline{Y}
\right]^{T},
\\
\widetilde{Y}_{0}^{(m)}=\left[ 0_{\mathbb{R}^{N+1}}\right] \mbox{
for all } m\in \mathbb{Z}_{+}\backslash\{0,1\},
\end{array}
\right.   \label{f17}
\end{equation}
are solvable for the $G_{n,t}$.

Note that $\widetilde{Y}$ (resp. $\widetilde{Y}_{ave})$ corresponds
to $\tau=1$ in (\ref{f15b}) (resp. (\ref{f14b})), so if
$(\varphi_{1,t})^{\ast}\widetilde{Y}=\widetilde{Y}_{ave}$ for every
non-negative real $t$, then by Proposition \ref{suspensiontool}, the
one-parameter family $(\pi\circ \varphi_{1,t} \circ
\mathcal{I}_t)_{t\in \mathbb{R}^{+}}$ solves the equation
(\ref{bigproblem}). We solve Lie's equations (\ref{f17}) in the next
section and we will show how this solvability will allow us to
compute formally a solution of (\ref{bigproblem}) in \S 2.3.

\subsection{Solving  Lie's equations associated to the equivalence problem (\ref{bigproblem})-
 first-order averaged system}
\label{SolvingLie}
 We treat in this section the problem of
solvability of Lie's equations associated with (\ref{bigproblem})
for $\widetilde{Z}=\widetilde{Y}_{ave}$.

In that respect, let  $M$ be  the map from
$\mathbb{Z}_{+}^{*}\times\mathbb{Z}_{+}^{*}$ to the set of the
vector fields of $\mathbb{R}^{N+1}$ defined by
\begin{equation}\label{defM}
\left\{
\begin{array}{l}
M(j+k,j)=\widetilde{Y}_{k}^{(j)},\,\forall \,(k,j)\, \in
\mathbb{Z}_{+}\times \mathbb{Z}_{+}^{*},\\
M(i,j)=0_{\mathbb{R}^{N+1}} \qquad i<j,
\end{array}
\right.
\end{equation}
where the terms $\widetilde{Y}_{k}^{(j)}$ are calculated from those
of (\ref{f15b}) and (\ref{formuledonnees}) by the recursive formula
(\ref{f10}).

We have the following result:
\begin{Prop}\label{propzero} For every integer $j$ greater than
or equal to two  and for every integer $l$ belonging to
$\{0,..,j-2\}$, we have $M(j,j-l)=0.$
\end{Prop}

\begin{proof}We prove this fact by recurrence. First we note that
$M(2,2)=\widetilde{Y}_{0}^{(2)}=0$ by (\ref{f17}).
 If $j=3$, then the recurrence formula (\ref{f10}) applied for $n=0$
and $i=2$ for the vector field $\widetilde{Y}$ and re-written for
$M$, gives $M(3,2)=M(3,3)-L_{\widetilde{G}_0}M(2,2)$ but
$M(3,3)=M(2,2)=0$ because of (\ref{f17}), so $M(3,2)=0$.

We make the following hypothesis, $\mathcal{P}(j)$, until a $j>3$,
$$\mathcal{P}(j):\,M(j-1,(j-1)-l)=0,
 \quad \mbox{for all }\, l \in \{0,..,j-3\}.$$
Then, as a consequence of this predicate, we get
$$ \forall \, p \in \{0,..,l\};\,M((j-1)-l+p,(j-1)-l)=0.$$

We show the heredity of the predicate. The formula (\ref{f10}) for
$n=l$ and $i=j-(l+1)$, takes the form
$$\widetilde{Y}_{l}^{(j-l)}
=\widetilde{Y}_{l+1}^{(j-(l+1))}+\sum_{p=0}^{p=l}C_l^{p}L_{\widetilde{G}_{l-p}}\widetilde{Y}_{p}^{(j-(l+1))},$$
so is re-written, for $M$,
$$M(j,j-l)=M(j,j-(l+1))+\sum_{p=0}^{p=l}C_l^{p}L_{\widetilde{G}_{l-p}}M(j-(l+1)+p,j-(l+1)),$$
and, because of the predicate,
\begin{equation}\label{propazero}
M(j,j-l)=M(j,j-1-l)\,\mbox{ for all }\, l \in \{0,..,j-3\} .
\end{equation}
But for $l=0$, $M(j,j)=\widetilde{Y}_{0}^{(j)}=0$ according to
(\ref{f17}), so $M(j,j-l)=0$ for every integer $l$ between $0$ and
$j-3$. Making $l=j-3$ in (\ref{propazero}) we get $M(j,2)=0$ and
therefore $\mathcal{P}(j+1)$.
\end{proof}

We make a useful remark for the clarity of some coming computations.

\begin{Rmk} Let $Y$ be the vector field associated with the system (1.1),
with an expanded $t$-suspended form given by (2.9) and (2.10). Let
$G$ be a vector field on $\mathbb{R}^N$ with $\widetilde{G}$ given
on the model of (2.12) and (2.13).  Let $\widetilde{Y}^{(m)}_j$ be
the terms calculated by the recursive formula (7.5) with
$A=\widetilde{Y}$ and $H=\widetilde{G}$. Then $\pi\circ
\widetilde{Y}^{(m)}_j=Y^{(m)}_j$ for every
$(m,j)\in\mathbb{Z}_{+}\times \mathbb{Z}_{+}$. \end{Rmk}

This last proposition allows us to solve Lie's equations easily in
our case, which we can summarize in the following
\begin{Prop}\label{solveLie}
Suppose that $\varphi_{1,t}$ generated by $\widetilde{G}_t$ given by
(\ref{tauserieG}), satisfies $(\varphi_{1,t})^{\ast}
\widetilde{Y}=\widetilde{Y}_{ave}$ for all $t\geq 0$. Then the
sequence $(G_j,\widetilde{Y}_j^{(1)}=\left[\alpha_{j}^{(1)},
Y_{j}^{(1)}\right]^{T})_{j \in \mathbb{Z}_{+}^{*}}$ (with
$\alpha_{j}^{(1)} \in \mathbb{R}$) is entirely determined by the
sequence $(\tilde{Y}_j^{(0)})_{j \in \mathbb{Z}_{+}^{*}}$ given by
(\ref{formuledonnees}) and (\ref{f15b}).

More exactly we have, for every positive integer $j$,
\begin{equation}\label{formuleGj}
\alpha_{j}^{(1)}=0 \mbox{,
}G_{j}(t)=\int_{0}^{t}(jL_{G_{j-1}}Y-Y_j^{(1)})ds\,,\,\forall \,t\,
\in \mathbb{R}^{+}
\end{equation}
{\em modulo} a constant vector, and,
\begin{equation}\label{formuleY1j}
\widetilde{Y}_{j}^{(1)}=-\sum_{k=0}^{k=j-1}C_{j-1}^{k}L_{\widetilde{G}_{k}}\widetilde{Y}_{j-1-k}^{(1)}.
\end{equation}
The initial term $(G_0,\widetilde{Y}_0^{(1)})$ is defined by the
first Lie equation and (\ref{f17}).
\end{Prop}

\begin{proof} First we prove that $\widetilde{G}_{0}$ (and
 $G_0$) is determined as follows.

By definition of the Lie derivative in $\mathbb{R}^{N}$, the first
Lie equation (cf. Appendix 1) takes the form
\begin{equation}  \label{f18}
\widetilde{Y}_{0}^{(1)}=L_{\widetilde{G}_{0}}\widetilde{Y}_{0}^{(0)}+\widetilde{Y}_{1}^{(0)}=D\widetilde{Y}_{0}^{(0)}.\widetilde{G}_{0}-D\widetilde{G}_{0}.\widetilde{Y}_{0}^{(0)}+\widetilde{Y}_{1}^{(0)},
\end{equation}
where $\widetilde{G}_{0}$ is the unknown. Since
$D\widetilde{Y}_{0}^{(0)}=\left[ 0_{\mathbb{R}^{(N+1)\times
(N+1)}}\right]$, (\ref{f18}) becomes
\begin{equation}  \label{f19}
\widetilde{Y}_{0}^{(1)}-\widetilde{Y}_{1}^{(0)}=-D\widetilde{G}_{0}.\widetilde{Y}_{0}^{(0)}.
\end{equation}
where
\begin{equation}  \label{f20}
D\widetilde{G}_{0}.\widetilde{Y}_{0}^{(0)}=\left[ 0_{\mathbb{R}}
,\partial _{t}G_{0} \right]^{T}.
\end{equation}

Hence, the equation (\ref{f18}) for $G_{0}$ becomes, taking into
account the expressions of $\widetilde{Y}_{1}^{(0)}$ and
$\widetilde{Y}_{0}^{(1)}$ given by (\ref{f15b}) and (\ref{f17}),
\begin{equation}  \label{f21}
\partial _{t}G_{0}(t,\xi )=Y(t,\xi )-\overline{Y(}\xi ).
\end{equation}
which gives $G_0$ by integration.

We obtain $\widetilde{Y}_1^{(1)}$ as follows. By (\ref{f10}) we get
$$ \widetilde{Y}_0^{(2)}=\widetilde{Y}_1^{(1)}+L_{\widetilde{G}_0}
\widetilde{Y}_{0}^{(1)},$$ but $\widetilde{Y}_{0}^{(1)}$ is given in
(\ref{f17}) and $\widetilde{Y}_0^{(2)}$ is equal to zero. Thus we
get $\widetilde{Y}_1^{(1)}=-L_{\widetilde{G}_0}
\widetilde{Y}_{0}^{(1)}$, which yields (\ref{formuleY1j}) for $j=1$.

Computing $\widetilde{G}_1$ by $(\ref{f10})$ gives:
$$
\widetilde{Y}^{(1)}_1 =
\widetilde{Y}^{(0)}_2+L_{\widetilde{G}_1}\widetilde{Y}^{(0)}_0+L_{\widetilde{G}_0}\widetilde{Y}^{(0)}_1,
$$
and, after easy manipulations as
$\widetilde{Y}^{(0)}_2=0_{\mathbb{R}^{N+1}}$ by
(\ref{formuledonnees}),
\begin{align*}
\alpha_{1}^{(1)}&=0,
\\
\partial_t G_1&= L_{G_0} (Y+\overline{Y}),
\end{align*}
yielding (\ref{formuleGj}) for $j=1$ after integration.

We consider the following predicate, $\mathcal{H}(j)$, for $j\geq
1$:

\begin{equation*}\mathcal{H}(j)\mbox { : } \left\{
\begin{array}{l}

\alpha_{j}^{(1)}=0 \mbox{,
}G_{j}(t)=\int_{0}^{t}(jL_{G_{j-1}}Y-Y_j^{(1)})ds\,,\,\forall \,t\,
\in \mathbb{R}^{+}, \\\mbox{ and }\\
\widetilde{Y}_{j}^{(1)}=-\sum_{k=0}^{k=j-1}C_{j-1}^{k}L_{\widetilde{G}_{k}}\widetilde{Y}_{j-1-k}^{(1)}.
\end{array}\right.
\end{equation*}

We prove $\mathcal{H}(j)$ by recurrence. The preceding computations
show $\mathcal{H}(1)$. Let $j \geq 2$. Suppose $\mathcal{H}(j-1)$.
Again (\ref{f10}) gives,
\begin{equation}\label{eqGj}
\widetilde{Y}_j^{(1)}=\widetilde{Y}_{j+1}^{(0)}+L_{\widetilde{G}_{j}}\widetilde{Y}_{0}^{(0)}+\sum_{k=1}^{k=j}C_j^{k}L_{\widetilde{G}_{j-k}}\widetilde{Y}_{k}^{(0)}
\end{equation}
and, as previously, we have
\begin{equation}\label{dtGj}
L_{\widetilde{G}_{j}}\widetilde{Y}_{0}^{(0)}=\left[0_{\mathbb{R}},-\partial_{t}G_{j}
\right]^{T}.
\end{equation}

The terms $\widetilde{Y}_k^{(0)}$  are equal to zero for all $k \geq
2$ (see (\ref{formuledonnees})), so, using (2.12), (2.13),
(\ref{eqGj}), and $Y_1^{(0)}=Y$, we get by integration the two first
assertions of $\mathcal{H}(j)$, noting that the first component of
the r.h.s of (\ref{eqGj}) is equal to zero (that gives
$\alpha_j^{(1)}=0$).

Finally, we obtain the last assertion of $\mathcal{H}(j)$ as
follows. From (\ref{f10}), we have for every positive integer $j$:
\begin{equation}\label{Ytilde2j}
\widetilde{Y}_{j-1}^{(2)}=\widetilde{Y}_{j}^{(1)}+\sum_{k=0}^{k=j-1}C_{j-1}^{k}L_{\widetilde{G}_{k}}\widetilde{Y}_{j-1-k}^{(1)},
\end{equation}
but by Proposition \ref{propzero} we get in particular that
$M(2+j-1,2)=\widetilde{Y}_{j-1}^{(2)}=0$ for every positive integer
$j$, so we deduce the last assertion of $\mathcal{H}(j)$ from
(\ref{Ytilde2j}). Thus $\mathcal{H}(j)$ is a necessary condition of
$\mathcal{H}(j-1)$, that completes the demonstration. \end{proof}

\subsection{Formal solution of the equivalence problem (\ref{bigproblem})}\label{solbigprob}
We describe in this section a procedure for obtaining formal series
representation of a solution of (\ref{bigproblem}), assuming that
Lie's equations are solved. In other words,  we assume, for this
section, that Lie's equations associated with the problem of
equivalence under consideration are solved.

Introduce for any sufficiently smooth vector fields
$F,H:\,(\tau,\xi)\in \mathbb{R}^{N+1}\rightarrow \mathbb{R}^{N}$,
\begin{equation}\label{deflambda}
\Lambda_{H} F=\frac{\partial F}{\partial \tau}+D_{\xi}F\cdot H,
\end{equation}
where  $D_{\xi}F$ stands for the usual Jacobian of $F$ in the local
coordinates $\xi\in \mathbb{R}^{N}$.

Naturally,  $\Lambda_{H}^{n}$ represents the iterated operation
$$\Lambda_{H}^{n}=\underbrace{\Lambda_{H} \circ .....\circ
\Lambda_{H}}_{\mbox{n}}.$$

The following lemma is the first step toward efficiently describing
$\mathcal{P}^{\infty}_{diff}$--corres\-pondence between a solution
$x$ of (\ref{generalsystem}), and one $z$ of $\dot{y}=Z(y)$.
\begin{Lem}\label{diffrepres}
 Let $\phi_{1,t}$ generated by the $\tau$-suspended vector
field $\widetilde{H}_t=(1,H_t)^{T}$ such that
$(\phi_{1,t})^{\ast}\widetilde{Y}=\widetilde{Z}$ for all $t\geq 0$,
then $\pi\circ\phi_{1,t}\circ\mathcal{I}_t$ is given as follows
\begin{equation}\label{correxbartox}
\pi\circ\phi_{1,t}\circ\mathcal{I}_t=Id_{\mathbb{R}^{N}}\!+\!
H_{0,t}+\sum_{i\geq 1}\frac{1}{(i\!+\!1)!} \left.
\Lambda_{H_t}^{i}H_t \right|_{\tau=0}, \mbox{ for all t
}\in\mathbb{R}^{+}.
\end{equation}
\end{Lem}
\begin{proof}
This lemma is a consequence of the Taylor formula. Let $z$ be
an integral curve of $\dot{y}=Z(y)$, and $x$ the solution of
(\ref{generalsystem}) being  in
$\mathcal{P}^{\infty}_{diff}$-correspondence  with  $z$ by
$(\pi\circ\phi_{1,t}\circ\mathcal{I}_t)_{t\in \mathbb{R}^{+}}$. We
have formally for $\tau=1$:
\begin{equation}  \label{formuletaylor}
(\phi_{1,t}\circ\mathcal{I}_t)(z(t))=\phi
(1,t,(t,z(t)))=(t,z(t))^T+\left. \frac{d\phi }{d\tau }\right|
_{\zeta}+\sum_{n\geq 1}\frac{1}{(n+1)!} \left. \frac{ d^{n+1}\phi
}{d\tau ^{n+1}}\right|_{\zeta}
\end{equation}
where $\zeta$ means the triplet $(\tau=0,t,z(t))$.

Consider the projection $\pi$ on the $N$ last components (projection
which commutes with the operator $\frac{d}{d\tau}$). If we compose
(\ref{formuletaylor}) to the left by $\pi$
 we obtain:
\begin{equation}
\begin{split}
(\pi \circ\phi_{1,t}\circ\mathcal{I}_t)(z(t))&=(\pi
\circ\phi_{1,t}) (t, z(t))\\
&=z(t)+\left. \frac{ d(\pi\circ\phi)}{d\tau}\right| _{\zeta}
+\sum_{n\geq 1}\frac{1}{(n\!+\!1)!} \left. \frac{
d^{n+1}(\pi\circ\phi)}{d\tau ^{n+1}}\right|_{\zeta}
\end{split}
\end{equation}
 taking into account that $\left.\frac{ d}{d\tau } (\pi\circ\phi)\right| _{\zeta}=H_0(t,z(t))$,
 it is therefore sufficient to prove for every integer $n$ greater or
equal than $1$ that:
\begin{equation}\label{formulereculambda}
 \left.
\frac{ d^{n+1}(\pi\circ\phi) }{d\tau ^{n+1}}\right|
_{\zeta}=\left.\Lambda_{H_t}^{n}H_t(\tau,\xi)\right|_{\zeta}.
\end{equation}

We prove (\ref{formulereculambda}). Using the chain rule it is
merely an exercise in basic calculus.

By assumptions,
\begin{equation}\label{phieq}
\frac{ d\phi_{\tau,t} }{d\tau }=\widetilde{H}_{t,\tau} \circ
\phi_{\tau,t}.
\end{equation}

Let us introduce the obvious notation
$\phi_{\tau,t}=(Id_{\mathbb{R}};\phi_{\tau,t}^{1};...;
\phi_{\tau,t}^{N})^{T}$ and,  denoting the $i^{th}$ component of
$H_{t,\tau}$ by $H_{t,\tau}^{i}$, the notation
$\widetilde{H}_{t,\tau}=(1,H_{t,\tau}^{1},...,H_{t,\tau}^{N})^{T}$.

Moreover, let us recall that the classical pullback by a
diffeomorphism $\Phi$ of a map $f: \mathbb{R}^{N} \to \mathbb{R}$ is
defined by:
\begin{equation}\label{pullfct}
\Phi^{\ast}f=f\circ \Phi,
\end{equation}
thus (\ref{phieq}) can be viewed component-wise as:
\begin{equation}\label{pullphi}
\frac{ d\phi_{\tau,t}^{i}}{d\tau }=
(\phi_{\tau,t})^{\ast}H_{t,\tau}^{i},\mbox{ for all } i \in \lbrace
1,...,N\rbrace.
\end{equation}

Noting that
$(\phi_{\tau,t})^{\ast}H_{t,\tau}^{i}(\cdot)=H^i_{t,\tau}\circ
\phi_{\tau,t}(\cdot)=H_t^i(\tau, \phi_{\tau,t}(\cdot))$, by
application of the chain rule, taking into account (\ref{pullfct}),
we obtain finally the key formula:
\begin{equation}\label{Keyformula2}
\frac{d}{d\tau}((\phi_{\tau,t})^{\ast}H_{t,\tau}^{i})=(\phi_{\tau,t})^{\ast}(\partial_{\tau}H_{t,\tau}^{i}
+\langle\nabla H_{t,\tau}^{i}, H_{t,\tau}\rangle),
\end{equation}
where $\nabla $ stands for the gradient  and
$\langle\cdot,\cdot\rangle$ for the usual Euclidean scalar product
on $\mathbb{R}^{N}$.

By introducing the following operator, which acts on families of
maps from $\mathbb{R}^{N}$ to $\mathbb{R}$:

\begin{equation}
\Upsilon_{H_t}(\cdot)=\partial_{\tau}(\cdot) +
\langle\nabla(\cdot),H_t\rangle,
\end{equation}
we can write (\ref{Keyformula2}) in the form:
\begin{equation}
\frac{d}{d\tau}((\phi_{\tau,t})^{\ast}H_{t,\tau}^{i})=(\phi_{\tau,t})^{\ast}\Upsilon_{H_{t,\tau}}(H_{t,\tau}^{i}),
\end{equation}
and thus, by induction it is clear that, for every integer $n$,
\begin{equation}
\frac{d^{n}}{d\tau^{n}}((\phi_{\tau,t})^{\ast}H_{t,\tau}^{i})=
(\phi_{\tau,t})^{\ast}\Upsilon_{H_{t,\tau}}^{n}(H_{t,\tau}^{i}),
\end{equation}
which, according to (\ref{pullphi}), leads to
\begin{equation}
\frac{d^{n+1}}{d\tau^{n+1}}(\phi_{\tau,t}^{i})=(\phi_{\tau,t})^{\ast}\Upsilon_{H_{t,\tau}}^{n}(H_{t,\tau}^{i}),
\end{equation}
and taking this expression at $\tau=0$ we get, because
$\phi_{0,t}=Id_{\mathbb{R}^{N+1}}$:
\begin{equation}\label{coeffcomp}
\frac{d^{n+1}}{d\tau^{n+1}}(
\phi_{\tau,t}^{i})\mid_{\tau=0}=\Upsilon_{H_{t,\tau}}^{n}(H_{t,\tau}^{i})\mid_{\tau=0},
\mbox{ for all }i \in \lbrace1,\cdots,N\rbrace.
\end{equation}

Let $F$ be a map from $\mathbb{R}^{N+1}$ to $\mathbb{R}^{N}$. By
definitions (\ref{pullfct}) and (\ref{deflambda}) we see that
\begin{equation}\label{lienope}
\Lambda_{H_t}F=(\Upsilon_{H_t}F^{1},\cdots,\Upsilon_{H_t}F^{N})^{T},
\end{equation}
so if we apply (\ref{lienope}) to $F=H_t$, taking into account
(\ref{coeffcomp}), we obtain:
\begin{equation}
\frac{d^{n+1}}{d\tau^{n+1}}(\pi \circ
\phi_{\tau,t})\mid_{\tau=0}=\Lambda_{H_t}^{n}H_t\mid_{\tau=0},
\end{equation}
which is (\ref{formulereculambda}).

\end{proof}

Consider now $H_{t}$ and $F_{t}$ time-dependent vector fields acting
on $\mathbb{R}^{N}$, both admitting $\tau$-series representations.
Omitting the dependence on $t$ for the vector fields in the
expansion of $F_{t}$ and $H_{t}$, we consider the following
recurrence formula:
\begin{equation}\label{recu3}
F_{n}^{[i+1]}=F_{n+1}^{[i]}+\sum_{k=0}^{n}C_{n}^{k}D_{\xi}F_{k}^{[i]}
\cdot H_{n-k};\mbox{ for all }i \in \mathbb{Z}_{+},
\end{equation}
initialized by $F_{n}^{[0]}=F_{n,t}$, $ F_{n,t}$ being the $n^{th}$
term of the $\tau$-series of $F_{t}$; and where $H_{n-k}$ is the
$(n-k)^{th}$ term in the $\tau$-series of $H_{t}$.

With these notations, we formulate, \vspace{1ex}

 \noindent{\bf Definition
2.8.}\label{defHtrans} {\em Let $t\in \mathbb{R}^{+}$ fixed. The
$H_{t}$-transformation evaluated at $\tau$, denoted by
$\mathcal{T}_{H_{t}}(\tau)$, of a time-dependent vector field
$F_{t}$ acting on $\mathbb{R}^{N}$, by a time-dependent vector field
$H_{t}$ acting on the same space, both admitting formal
$\tau$-series, is given by:
\begin{equation}\label{Ttrans}
\mathcal{T}_{H_{t}}(\tau)\cdot
F_{t}=\sum_{n\geq0}\frac{\tau^{n+1}}{(n+1)!}F_{0,t}^{[n]},
\end{equation}
where each $F_{0,t}^{[n]}$ is calculated  using the recurrence
formula (\ref{recu3}), initialized by $F_{n,t}^{[0]}=F_{n,t}$; $
F_{n,t}$ being the $n^{th}$ term of the formal $\tau$-series of
$F_{t}$. When $F=H$, the term $F_{0,t}^{[n]}$ will be called the
$(n\!+\!1)^{th}$-corrector.}
 %\refstepcounter{theorem}
\vspace{1ex}

 \begin{Rmk} It is important to note here
that $\mathcal{T}_{H_{t}}(0)\cdot F_{t}=0$, allowing near-identity
transformations for small $\tau$ (cf. Theorem \ref{diffeorepre}).
\end{Rmk}

 We then have the following theorem:
\begin{Thm}\label{diffeorepre}
If $\phi_{\tau,t}$ is generated by a time-dependent vector field of
the form $(1,H_{t})^{T}$ on a region $\mathbb{R}^{+}\times\Omega$,
for which $H_t$  has a formal $\tau$-series, $\sum_{n\geq
0}\frac{\tau^{n}}{n!}H_{n,t}$, on this region for all $\tau\in I$;
$I$ being an interval of reals such that $0\in I$, then:
\begin{equation}\label{soloderepre}
\pi\circ\phi_{\tau,t}\circ\mathcal{I}_t=Id_{\mathbb{R}^{N}}+
\mathcal{T}_{H_{t}}(\tau)\cdot H_{t}, \mbox{ on }  I\times
\mathbb{R}^{+}\times\Omega.
\end{equation}
\end{Thm}

\begin{proof}
First of all, note that the derivations in the proof of Lemma
\ref{diffrepres} show that, for all $t \in\mathbb{R}^{+}$ and
$\tau\in I$,
\begin{equation}\label{1step}
\pi\circ\phi_{\tau,t}\circ\mathcal{I}_t=Id_{\mathbb{R}^{N}}+\tau\cdot
H_{0,t}+\sum_{i\geq 1}\frac{\tau^{i+1}}{(i+1)!} \left.
\Lambda_{H_t}^{i}H_t \right|_{\tau=0}.
\end{equation}
Omitting the time-dependence, simple calculus shows
\begin{equation}
\Lambda_{H}H =\sum_{n\geq 0}\frac{\tau ^{n}}{n!}(H_{n+1}+
\sum_{k=0}^{n}C_{n}^{k}D_{\xi}H_{k}\cdot H_{n-k}),
\end{equation}
and defining the vector fields $H_{n}^{[1]}$ as
\begin{equation}\label{initrecu2}
H_{n}^{[1]}=H_{n+1}+ \sum_{k=0}^{n}C_{n}^{k}D_{\xi}H_{k}\cdot
H_{n-k},
\end{equation}
following a procedure similar to the one of \cite{her}, but for the
operator $\Lambda_{H}$, we get the recurrence formula:

\begin{equation}  \label{recu2}
H_{n}^{[i+1]}=H_{n+1}^{[i]}+\sum_{k=0}^{n}C_{n}^{k}D_{\xi}H_{k}^{[i]}\cdot
H_{n-k};\mbox{ for all }i \in \mathbb{Z}_{+},
\end{equation}
where $H_{n}^{[i]}$ is the n$^{th}$ term of the series
\begin{equation}  \label{LambdaiG}
\Lambda_{H}^{i} H=\sum_{n\geq 0}\frac{\tau ^{n}}{n!}H_{n}^{[i]},
\end{equation}
and the initial terms are given by $H_{n}^{[0]}=H_n$ for every
 non-negative integer $n$.

We conclude therefore that:
\begin{equation}\label{icorrector}
\Lambda_{H}^{i} H\mid_{\tau=0}=H_{0}^{[i]},
\end{equation}
which proves the theorem after taking into account Definition
2.8.\end{proof}

In practice Theorem \ref{diffeorepre} and Corollary 7.2, permit to
determine formally a solution of (\ref{bigproblem}). For example,
when the ``target" autonomous field is $\overline{Y}$, as shows
Proposition \ref{solveLie}, Lie's equations (\ref{f17}) are
solvable, thus $(Id_{\mathbb{R}^{N}}+ \mathcal{T}_{G_{t}}(1)\cdot
G_{t})_{t\in\mathbb{R}^{+}}$ can be computed using computer algebra
implementation \cite{b,ko2,ra} and thereby gives a formal solution
of (\ref{bigproblem}) (with $Z=\overline{Y}$), where the generator
$G$ is determined by (\ref{formuleGj}) in Proposition
\ref{solveLie}.

\section{New higher-order averaged systems and
related corrections}\label{higherorder}

Higher-order averaged systems are linked to the way change of
variables are performed. In the $\epsilon$-dependent case, the usual
formulae are obtained by using formulae for arbitrary-order
derivatives of compositions of differentiable vector functions
\cite{fra}, on the one hand, and by using the implicit function
theorem, on the other
 \cite{esd, mur, per}. This form for higher-order averaging is not the only
 one, as it is pointed out in \cite[pp. 36--37]{lo}. We describe
 here other higher-order averaged systems both for
 $\epsilon$-dependent and independent cases, via Lie transforms formalism, but different
 from those of \cite{na,jap}, and given by simple explicit
 formulae.

\subsection{Formal algorithm of higher-order averaging  based on Lie
transforms}\label{Falgo} We describe here a way of computing
formally the higher-order averaged systems according to the Lie
transform of the vector field $\widetilde{Y} \in
\mathcal{P}^{\infty}(\widetilde{\Omega})$ in consideration.

In that respect we consider first the parametric standard form
$\dot{x}=\tau Y(t,x)$, where $\tau$ is a real parameter. Then we
have the following proposition.
\begin{Prop}\label{nthaveraging}
Let $n$ be an integer greater than or equal to two. Let
$\overline{Y}^{(n)} \in\mathcal{C}^{\infty}(\Omega)$ be the $n^{th}$
term of the finite sequence satisfying
\begin{description}
\item [\textrm{$(a_{\tau})$}]
$\overline{Y}^{(i+1)}=\overline{Y}^{(i)}+\tau^{i+1}U_{i+1} $, where
$U_{i+1}$ is an autonomous field in $\mathcal{C}^{\infty}(\Omega)$
for all $i\in \{0,\cdots,n-1\}$, with $\overline{Y}^{(0)}=0$,
$U_1=\overline{Y}$ and $U_0=0$.

Assume that there exists $\phi_{\tau,t}$ generated by
$\widetilde{W}_t=[1,W_t]^{T}$ having a $\tau$-series representation
such that

\item [\textrm{$(b_{\tau})$}] For all $i\in \{1,\cdots,n-1\},$ the
$i^{th}$ term $W_i$ of the $\tau$-series representation of $W_t$ is
$T$-periodic.
\end{description}
If $\pi\circ\phi_{\tau,t}\circ\mathcal{I}_t$ satisfies for all
$t\in\mathbb{R}^{+}$
\begin{equation}\label{lienmoyneps}
(\pi\circ\phi_{\tau,t}\circ\mathcal{I}_t)^{-1} \ast(\tau Y_t) +
\partial_t(\pi\circ\phi_{\tau,t}\circ\mathcal{I}_t)^{-1}\circ(\pi\circ\phi_{\tau,t}\circ\mathcal{I}_t)=
\overline{Y}^{(n)},
\end{equation}
then
\begin{equation}\label{genenthaveragedsyst}
\overline{Y}^{(n)}=\tau
\overline{Y}+\sum_{m=2}^{n}\frac{\tau^m}{m!}Y_{0}^{(m)},
\end{equation}
where for $m \in \{2,\cdots,n\}$
\begin{equation}\label{eltn}
m!U_m=Y_{0}^{(m)}=\frac{1}{T}\int_{0}^{T}(\sum_{l=0}^{m-2}(\sum_{k=0}^{k=l}C_{l}^{k}L_{W_{l-k}}Y_{k}^{(m-1-l)})+
C_{m-1}^{1}L_{W_{m-2}}Y)ds;
\end{equation} and $W_0=G_0$, with $G_0$ given from (\ref{f21}).
\end{Prop}

\begin{proof} Let an integer $n\geq 2 $. If there exists $\phi_{\tau,t}$
generated by $\widetilde{W}_t=[1_{\mathbb{R}},W_t]^{T}$ having a
$\tau$-series representation on the model of (\ref{tauserieG}) and
(\ref{Gi}), such that for every positive real $t$, the map
$\pi\circ\phi_{\tau,t}\circ \mathcal{I}_{t}$ satisfies
(\ref{lienmoyneps}), then according to Proposition
\ref{suspensiontool} the equivalence problem (\ref{lienmoyneps})
takes the form of the pullback one $\phi_{\tau,t}^{-1}\ast (\tau
\widetilde{Y})=\widetilde{\overline{Y}^{(n)}}$, that gives by
Theorem 7.1
\begin{equation*}
L(\widetilde{W_t})(\tau)\cdot(\tau\widetilde{Y})=\widetilde{\overline{Y}^{(n)}},
\end{equation*}
where $\overline{Y}^{(n)}$ is given by assumption ($a_{\tau}$).

If we applied (7.4) in this case, we get:
\begin{equation*}
\sum_{m=0}^{m=\infty}\frac{\tau^m}{m!}\widetilde{Y}^{(m)}_{0,t}=\sum_{m=0}^{m=n}\tau^m\widetilde{U}_m,
\end{equation*}
according to the assumption $(a_{\tau})$, where for $m\geq1$,
$\widetilde{U}_m$ (resp. $\widetilde{Y}^{(m)}_{0,t}$) is the flat
form of $U_m$ (resp. $Y^{(m)}_{0,t}$), and the $t$-suspended form
for $m=0$.

By projecting onto the  $N$-last components of $\mathbb{R}^{N+1}$,
by taking into account the remark before Proposition 2.6 we get for
every $m\in\{2,\cdots,n\}$,
\begin{equation}\label{formUi}
U_{m}=\frac{1}{m!}Y_{0}^{(m)},
\end{equation}
giving thereby (3.2).

Express the term $Y_{0}^{(m)}$. By (\ref{f10}) we have for
$m\in\{2,\cdots,n\}$:
\begin{equation}\label{source}
Y_{0}^{(m)}=Y_{1}^{(m-1)}+L_{W_0}Y_{0}^{(m-1)},
\end{equation}
expressing $Y_{1}^{(m-1)}$ by (\ref{f10}) we get:

\begin{equation}
Y_{1}^{(m-1)}=Y_{2}^{(m-2)}+\sum_{k=0}^{1}
C_1^kL_{W_{1-k}}Y_{k}^{(m-2)},
\end{equation}
and  continuing the procedure to an arbitrary rank $l \in
\{2,\cdots,m-1\}$:
\begin{equation}\label{brik}
Y_{l}^{(m-l)}=Y_{l+1}^{(m-1-l)}+\sum_{k=0}^{l}
C_l^kL_{W_{l-k}}Y_{k}^{(m-1-l)},
\end{equation}
thus collecting, in (\ref{source}), from $l=1$ until $l=m-1$
successively the terms $Y_{l+1}^{(m-1-l)}$  we get:
\begin{equation}\label{inter}
Y_{0}^{(m)}=Y_m^{(0)}+\sum_{l=0}^{m-1}(\sum_{k=0}^{k=l}C_{l}^{k}L_{W_{l-k}}Y_{k}^{(m-1-l)}).
%\sum_{k=1}^{m-1}C_{m-1}^{k}L_{W_{m-1-k}}Y_{k}^{(0)})
\end{equation}

Noting that $Y_{1}^{(0)}=\pi\circ\widetilde{Y}^{(0)}_{1}$ is equal
to $Y$ (see (\ref{f15b})) and taking into account
(\ref{formuledonnees}), we get

\begin{equation}
\sum_{k=0}^{m-1}C_{m-1}^{k}L_{W_{m-1-k}}Y_{k}^{(0)}=C_{m-1}^{1}L_{W_{m-2}}Y+L_{W_{m-1}}Y_0^{(0)},
\end{equation}
and noting  moreover that, as in (\ref{dtGj}),
\begin{equation}
L_{W_{m-1}}Y_0^{(0)}=-\partial_t W_{m-1},
\end{equation}
the formula (\ref{inter}) takes the form
\begin{equation}\label{inter2}
Y_{0}^{(m)}=\sum_{l=0}^{m-2}(\sum_{k=0}^{k=l}C_{l}^{k}L_{W_{l-k}}Y_{k}^{(m-1-l)})+
C_{m-1}^{1}L_{W_{m-2}}Y-\partial_t W_{m-1}.
\end{equation}

Integrating (\ref{inter2}) with respect to time, from $0$ to the
minimal period $T$, we obtain according to the fact that
$Y_{0}^{(m)}$ is autonomous by assumption ($a_{\tau}$) and the fact
that $W_{m-1}$ is $T$-periodic by assumption  ($b_{\tau}$), formula
(\ref{eltn}) taking into account (\ref{formUi}).

Finally, noting  that the first Lie equation is the same at
each-order (cf. (\ref{firsteq}) and related discussion), we easily
conclude that $W_0=G_0$, where $G_0$ is given by (\ref{f21}), which
completes the proof of this proposition.\end{proof}

Note that  when we talk about the $n^{th}$ averaged system, in any
cases, we will consider naturally the following system of
differential equations:
\begin{equation}\label{nthaveeq}
\dot{x}=\overline{Y}^{(n)}(x),
\end{equation}
where $\overline{Y}^{(n)}$ will be given by
(\ref{genenthaveragedsyst}) in Proposition \ref{nthaveraging}.

\subsection{Correction of higher-order averaged
systems and solvability of related Lie's equations}\label{corectave}

We adopt here the procedure described in \S \ref{solbigprob} for
correcting higher-order averaged systems.

We consider  the system of differential equations
(\ref{generalsystem}), and for a given $p\geq 2$ a higher-order
averaged form given by $\dot{x}=\overline{Y}^{(p)}(x)$ where
$\overline{Y}^{(p)}$ is given by (\ref{genenthaveragedsyst}) in
Proposition \ref{nthaveraging} for $\tau=1$.

Consider the diffeomorphism $\phi_{1,t}$ generated by
$\widetilde{W}_t=[1,W_t]^{T}$ in Proposition \ref{nthaveraging}.
Then, according to Theorem \ref{diffeorepre}, the transformation
$\pi\circ\phi_{1,t}\circ \mathcal{I}_{t}$ is given formally as the
following formal series representation
\begin{equation}\label{seriedepave}
Id_{\mathbb{R}^{N}}+\mathcal{T}_{W_{t}}(1)\cdot W_{t},
\end{equation}
that allows computer algebra implementation \cite{ra}, and where
$W_t$ depends on the averaged system $\overline{Y}^{(p)}$ (cf. again
Proposition \ref{nthaveraging}).

It is important to note here that for two given averaged systems,
$\overline{Y}^{(p)}$ and $\overline{Y}^{(q)}$, ($p \not=q$), the
corresponding generators,  say $W$ and $V$ for fixing the ideas,
related to each system are not {\em a priori} equal. Therefore,
considering Theorem 2.9 and Definition 2.8, the $W_t$- and
$V_t$-transformation given by analogous formulae to
(\ref{seriedepave}), are also not equal.

For instance for $p=1$ (resp. $p=2$), {\it i.e.} when $W=G$ (resp.
$W=\Gamma$), $G$ (resp. $\Gamma$) being defined as the generator of
time-dependent diffeomorphisms which transform $Y$ into
$\overline{Y}^{(1)}=\overline{Y}$ (resp.
$\overline{Y}^{(2)}=\overline{Y}+Y_0^{(2)})$; the second corrector
in $\mathcal{T}_{G_{t}}(1)\cdot G_{t}$ and in
$\mathcal{T}_{\Gamma_{t}}(1)\cdot \Gamma_{t}$ are respectively
\begin{equation}
G_{0,t}^{[1]}=G_{1,t}+DG_{0,t}.G_{0,t},
\end{equation}
which is given by (\ref{recu3}) with $F=H=G$, $i=0$, $n=0$  and,
\begin{equation}
\Gamma_{0,t}^{[1]}=\Gamma_{1,t}+D\Gamma_{0,t}.\Gamma_{0,t},
\end{equation}
which is  given by (\ref{recu3}) with $F=H=\Gamma$, $i=0$ and $n=0$.

For $p=2$, we have $\Gamma_{0,t}=G_{0,t}$ (see (\ref{firsteq})
below) and, by applying (\ref{f10}) for $i=1$, $n=0$, $A=\widetilde
{Y}$, and $H=\widetilde{\Gamma}=[1_{\mathbb{R}}, \Gamma]^T$, with
some easy computations we get
 \begin{equation}\label{oscpart}
\Gamma_{1,t}=\int_{0}^{t}\left( D(Y+\overline{Y}).\Gamma_{0}-D\Gamma
_{0}.(Y+\overline{Y})-Y_{0}^{(2)}\right) ds,
 \end{equation}
\textit{modulo} a constant, where $Y_{0}^{(2)}$ is given by
(\ref{eltn}) with $m=2$.

Same algebraic procedure for obtaining $\Gamma_{1,t}$ and related
computations, in the case $p=1$, leads to
\begin{equation}\label{calculG1}
G_{1,t}=\int_{0}^{t}\left(
D(Y+\overline{Y}).G_{0}-DG_{0}.(Y+\overline{Y})\right) ds,
\end{equation}
\textit{modulo} a constant.

Therefore, we see that the $G_t$-transformation and the
$\Gamma_t$-transformation are different. Indeed by substituting
(3.17) and (3.16) respectively in (3.14) and (3.15), we note that
the second correctors $G_{0,t}^{[1]}$ and $\Gamma_{0,t}^{[1]}$ are
distinct. Same remarks hold for higher-order correctors related to
two distinct higher-order averaged systems, allowing us to conclude
the non-invariance of every higher-order corrector.

We consider now the invariance of the first corrector. Indeed, the
relation (\ref{f10}) for $i=0$, $n=0$, $A=\widetilde{Y}$ and
$H_0=\widetilde{W}_0$ or $H_0=\widetilde{V_0}$ yields for the same
Lie's equation
\begin{equation}\label{firsteq}
\partial_t(\pi\circ H_{0,t})=Y_t-\overline{Y}
\end{equation}
after projection. This demonstrates that the first corrector term
$W_{0,t}^{[0]}$  is independent from the averaged system considered.

Note that the non-invariance (resp. invariance) of higher-order
(resp. first-order) corrector $W_{0,t}^{[n]}$ ($n\geq 1$) (resp.
$W_{0,t}^{[0]}$), was mentioned in \cite[p. 36]{lo} in the context
of general averaging ({\it i.e.} without the use of Lie transforms).

We conclude this section concerning the solvability of Lie's
equations related to higher-order averaged systems.

Fix $p\geq 2$. For computer implementation of (\ref{seriedepave}),
we have to initialize (\ref{recu3}) with $H=W=F$; in other words we
have to solve Lie's equations associated with (\ref{lienmoyneps}).
Those equations show a structure similar to that studied in \S
\ref{SolvingLie}. We outline how to solve them for $p >2$, and we
give the precise results for the case $p=2$. Here, the key point of
the Lie's equations solvability consists in extending Proposition
2.5 associated to the first-order equivalence problem ({\it i.e.}
for $n=1$ in (3.1)) to one associated to a higher-order equivalence
problem (3.1). Indeed, based on ($\ref{f10}$), by introducing a map
$\mathcal{N}$ on the model of (\ref{defM}) with elements of the
diagonal given by (3.3), we can show that a proposition analogous to
Proposition \ref{propzero} holds, by shifting of ($p-1$) ranges the
first column of vanishing terms in the infinite matricial
representation of $\mathcal{N}$. This allows us to compute the $W_j$
({\it i.e.} the $j^{th}$-term of the $\tau$-series representation of
$W$ in Proposition 3.1), by adapting computations giving the $G_j$
in Proposition \ref{solveLie}. We illustrate this procedure in the
case $p=2$ for the convenience of the reader.

In this latter case, we can indeed show the two propositions below.
First of all note that, by applying $(3.3)$ in Proposition 3.1 with
$m=2$, we have
\begin{equation}\label{secav}
Y_0^{(2)}=\frac{1}{T} \cdot \int_{0}^{T}(D(Y+\overline{Y}).\Gamma
_{0}-D\Gamma _{0}.(Y+\overline{Y}))ds,
\end{equation}
where $\Gamma_0$ is given by integration of the r.h.s of (3.18)
because of the invariance  of the first corrector.

The analogous of Proposition 2.5 associated to the second-order
equivalence problem ({\it i.e.} $n=2$ in (3.1)) is
\begin{Prop}
Let $\mathcal{N}$ be a map introduced on the model of (\ref{defM}),
such that $\mathcal{N}(1,1)=\widetilde{Y}_0^{(1)}$ is given by
(2.14), $\mathcal{N}(2,2)=\widetilde{Y}_0^{(2)}=[0, Y_0^{(2)}]^{T}$
where $Y_0^{(2)}$ is given by (\ref{secav}), and
$\mathcal{N}(j,j)=\widetilde{Y}_0^{(j)}$ are null for every integer
$j$ greater than or equal to three. Then for every such integer $j$,
and for every integer $l$ belonging to $\{0,...,j-3\}$, we have
$\mathcal{N}(j,j-l)=0$.
\end{Prop}

Following the model of the proof of Proposition 2.6 that was based
on Proposition 2.5, we can show, taking into account Proposition
3.2,
\begin{Prop}\label{Liesolves2ndave}
Suppose that $\psi_{1,t}$ generated by
$\widetilde{\Gamma}_t=[1_{\mathbb{R}},\Gamma_t]^T$ satisfies
$\psi_{1,t}^{-1}\ast \widetilde{Y}=\widetilde{\overline{Y}^{(2)}}$
for all $t\geq 0$. Then the sequence
$(\Gamma_j,\widetilde{Y}_j^{(1)},\widetilde{Y}_j^{(2)})_{j \in
\mathbb{Z}_{+}^{*}}$ is entirely determined by the sequence
$(\tilde{Y}_j^{(0)})_{j \in \mathbb{Z}_{+}^{*}}$ given by
(\ref{formuledonnees}) and (2.9). More exactly we have for every
positive integer $j$,
\begin{equation}\label{formulegammaj}
\Gamma_{j}(t)=\int_{0}^{t}(jL_{\Gamma_{j-1}}Y-Y_j^{(1)})ds\,\,,\,\forall
\,t\, \in \mathbb{R}^{+},
\end{equation}
modulo a constant vector,
\begin{equation}\label{formulegammaY1j}
\widetilde{Y}_{j}^{(1)}=\widetilde{Y}_{j-1}^{(2)}-\sum_{k=0}^{k=j-1}C_{j-1}^{k}L_{\widetilde{\Gamma}_{k}}\widetilde{Y}_{j-1-k}^{(1)},
\end{equation}
and
\begin{equation}\label{formulegammaY2j}
\widetilde{Y}_{j}^{(2)}=-\sum_{k=0}^{k=j-1}C_{j-1}^{k}L_{\widetilde{\Gamma}_{k}}\widetilde{Y}_{j-1-k}^{(2)},
\end{equation}
with every first component of $\widetilde{Y}_j^{(1)}$ and
$\widetilde{Y}_j^{(2)}$ being equal to zero. The initial term
$(\Gamma_0,\widetilde{Y}_0^{(1)},\widetilde{Y}_0^{(2)})$ is defined
by the first Lie equation, (\ref{f17}), and (3.19).
\end{Prop}
%The proof follows the same lines of the one of proposition 3.3 and
This concludes the solvability of Lie's equations related to the
second-order averaged system.

Such propositions can be extended for higher-order averaged systems,
that rules, by induction, the  solvability of related Lie's
equations.

\subsection{A heuristic comparison with the classical
approach}\label{heuristic}

We make in this section a few comments on previous works and ours as
well, and discuss a rigorous setting concerning formulae derived in
the preceding sections, in the case when $\tau$ is a small
parameter, usually denoted by $\epsilon$.

As explained in \S  \ref{corectave}, Lie's equations are solvable to
each order of the averaged systems, giving thereby the time-varying
change of coordinates by Theorem \ref{diffeorepre}. For instance,
for $n=2$ in Proposition \ref{nthaveraging}, we can solve Lie's
equations related to the second-averaged system (cf. Proposition
3.3). Consider the system (1.1), if we make the change of variables
\begin{equation}\label{changeofvariables}
x=y+\epsilon W_{0,t}(y)+ \frac{\epsilon^2}{2}W_{0,t}^{[1]}(y),
\end{equation}
(that is a truncation up to order $2$ of $\mathcal{T}_{W_t}(1)\cdot
W_t$), then following the procedure  of \cite[pp. 168-169]{gh} under
sufficient smoothness assumptions, we obtain that $y$ satisfies
\begin{equation}\label{newvariable}
\dot{y}=\epsilon\overline{Y}+\frac{\epsilon^2}{2}Y_0^{(2)}(y)+O(\epsilon^{3}).
\end{equation}
This shows that our procedure of averaging-correction is coherent
and rigorously justified up to order two, for $\epsilon$
sufficiently small.

Extending this kind of reasoning to order greater than $2$,
following, for example, the procedure described in \cite[\S3]{mur},
we can show that our procedure is rigorously justified to any order,
for $\epsilon$ sufficiently small, making thereby near-identity
transformations. This permits us to obtain convergence theorems on
$n^{th}$-order averaging, with suitable assumptions, such as those
used in \cite{per}, for instance. These kind of results increase
precision but not the time-scale of validity. The latter can be
improved if the system shows some attracting properties, as it is
classically done \cite{lo, mur, sv}.

From a practical point of view, we obtain more manageable formulae
than those usually obtained \cite{esd, per}. Indeed, our approach
based on Proposition \ref{nthaveraging} gives an efficient procedure
of construction of any higher-order averaged systems, a procedure
which can be implemented on a computer using a symbolic
computational software, as it is done for Lie transforms in general
\cite{b, ko2, ra}. Actually, many authors \cite{ch, mur, na} pointed
out the fact that Lie transforms can be used to obtain higher-order
averaged systems, but in our knowledge, a systematic explicit
exposition like the one given here has not been published (although
a related paper \cite{jap} presents some of the elements.) The fact
that the classical formulae are not manageable may be more evident
in \cite{har}, where the authors develop an algorithm of higher
order averaging for linear equations with periodic coefficients, in
order to simplify the one of \cite{esd}.

Finally, it is interesting to note that usually the
$i^{th}$-correctors are thought to be the anti-derivative in time of
the oscillatory part associated to each order (cf. (6) via (7) and
(5) of \cite{per} for example). This structural form is the same for
correctors $W_0^{[i]}$ in our computations, with $Y_0^{(i)}$ given
by (\ref{eltn}) instead of the classical averaged part of order $i$,
as a consequence of  Lie's equations (the reader will be convinced
by comparing (6) of \cite{per} for $i=2$ and (\ref{oscpart}), for
example).

\section{Choice of appropriate initial
data and the $m^{th}$ approximation of the $n^{th}$ type
}\label{seccondinit}

Let $n$ be a positive integer. Consider $\phi_1 \in
\mathcal{P}^{\infty}_{d}(\widetilde{\Omega})$, which satisfies the
hypothesis of Proposition \ref{nthaveraging}. Then, by \S
\ref{fiteq} and Proposition \ref{suspensiontool}, every solution $x$
of $Y$ through $x_0 \in \Omega$ is in
$\mathcal{P}^{\infty}_{diff}$-correspondence by
$(\pi\circ\phi_{1,t}\circ\mathcal{I}_t)_{t\in\mathbb{R}^{+}}$ to a
unique solution $\overline{x}^{(n)}$ of the $n^{th}$ averaged system
given by (\ref{nthaveeq}) in Proposition \ref{nthaveraging}. We
recall that the latter means
\begin{equation}\label{conjug}
x(t,x_0)=(\pi\circ \phi _{1,t}\circ \mathcal{I}_t)
(\overline{x}^{(n)}(t,(\pi\circ\phi_{1,0}^{-1}\circ\mathcal{I}_0)(x_0))),\mbox{
for all } t\in\mathbb{R}^{+}.
\end{equation}

In a number of applications $x_0$ is given, and thus,  in order to
compute an approximation of $x$ based on $\overline{x}^{(n)}$, we
have to determine the initial condition appropriate for
$\mathcal{P}^{\infty}_{diff}$-correspondence, that is
$(\pi\circ\phi_{1,0}^{-1}\circ\mathcal{I}_0)(x_0)$ as shown by
(\ref{conjug}).

This problem of determination of the inverse transformation for a
given Lie transformation is not new, and several algorithms have
been proposed (see e.g. \cite{hen, v}). We present here an approach
based on \cite{he}.

In order to be consistent with our notations used in
(\ref{tauserieG}) and (\ref{Gi}), we consider $K$ (resp. $H$) the
generator of the inverse (resp. direct) transformation given by its
formal $\tau$-series
\begin{equation}\label{K}
\widetilde{K}(t,\tilde{\xi})=\sum_{n \geq
0}\frac{\tau^{n}}{n!}\widetilde{K}_n(t,\xi), \mbox{ (resp.
}\widetilde{H}(t,\tilde{\xi})=\sum_{n \geq
0}\frac{\tau^{n}}{n!}\widetilde{H}_n(t,\xi)).
\end{equation}
In the context of our framework, Proposition $5$ of \cite{he} can be
reformulated as the following
\begin{Prop}\label{inversegenerator}
Let $\phi_{1,t}$ generated by $\widetilde{H}_t$ defined in (\ref{K})
and let $\phi^{-1}_{1,t}$ generated by the $\tau$-suspended vector
field $\widetilde{K}_t$, then:
\begin{equation}
\widetilde{K}_t=-(\phi_{1,t})^{ \ast} \widetilde{H}_t.
\end{equation}
\end{Prop}

\begin{proof} We provide a proof for the sake of completeness. Let us denote
the inverse of $\phi_{\tau,t}$ by $\eta_{\tau,t}$, {\it i.e.}, for
all $(\tau, t)\in\mathbb{R}\times\mathbb{R}^{+}$, and $x\in\Omega$,
$\phi_{\tau,t}(\eta_{\tau,t}(x))= {x}$.

Taking the derivative with respect to $\tau$ yields
$$
(D_{x}\phi_{\tau,t})(\eta_{\tau,t}(x))\frac{d \eta_{\tau,t}}{d
\tau}(x) + \frac{\partial\phi_{\tau,t}}{\partial
\tau}(\eta_{\tau,t}(x)) = 0 . \nonumber
$$

We express the second term on the left-hand side in terms of the
generator $\widetilde{H}_t$, and move it to the right-hand side to
obtain
$$
 (D_{x}\phi_{\tau,t})(\eta_{\tau,t}(x))\frac{d
\eta_{\tau,t}}{d
\tau}(x)=-\widetilde{H}_t(\phi_t(\tau,\eta_{\tau,t}(x))),
$$
which gives, after multiplying both sides with the inverse of
$(D_{x}\phi_{\tau,t})(\eta_{\tau,t}(x))$

$$
\frac{d \eta_{\tau,t}}{d\tau}(x) =-D_{x}\phi_{\tau,t}^{-1}(x)\cdot
\widetilde{H}_t(\phi_{\tau,t}(\eta_{\tau,t}(x))).
$$
The latter leads to,
$$
\frac{d \eta_{\tau,t}}{d \tau}(x)
=-D_{x}\phi_{\tau,t}^{-1}(\phi_{\tau,t}(\eta_{\tau,t}(x)))\cdot
\widetilde{H}_t(\phi_{\tau,t}(\eta_{\tau,t}(x))),
$$
that is, with  Definition 2.2, to,
$$
\frac{d \eta_{\tau,t}}{d
\tau}(x)=-(\phi_{\tau,t}^{-1}\ast\widetilde{H}_t)(\eta_{\tau,t}(x)).
$$

Therefore, for each $\tau$ and each $t$, we deduce that
$\eta_{\tau,t}$ is generated by\\
$-(\phi_{\tau,t})^{\ast}\widetilde{H}_t$, and thus substituting
$\tau=1$ the proof is complete.

\end{proof}

The following proposition, which is an obvious corollary of Theorem
\ref{diffeorepre} and Proposition \ref{inversegenerator}, where
$\mathcal{T}_{K_t}(1)$ is given by Definition 2.8, determines
formally the inverse transformation for all $t\geq 0$, and in
particular the one in (\ref{conjug}) for $t=0$.

\begin{Prop}\label{inverserepre}
Let $\phi_{1,t}$ generated by a time-dependent vector field of the
form $\widetilde{H}_{t}=(1,H_t)^{T}$ and let
$\widetilde{K}_t=-(\phi_{1,t})^{\ast}\widetilde{H}_{t}$, then for
every non-negative real $t$:
\begin{equation}\label{formrepreinverse}
\pi\circ\phi_{1,t}^{-1}\circ\mathcal{I}_t=Id_{\mathbb{R}^{N}}+\mathcal{T}_{K_t}(1)\cdot{K_t}.
\end{equation}
\end{Prop}

\vspace{1ex}

 This proposition allows us to compute the $m^{th}$
approximation of the solution of (\ref{generalsystem}), which is in
$\mathcal{P}^{\infty}_{diff}$-correspondence to one of a $n^{th}$
averaged system by the following algorithm which defines the {\em
$m^{th}$ approximation of the $n^{th}$ type}:

\begin{itemize}
\item[(1).] Compute the $n^{th}$ averaged system given by
(\ref{nthaveeq}) and Proposition \ref{nthaveraging}.

\item[(2).] Solve Lie's equations as explained in \S 3.2 until the order $m$,
      giving the $m$ first $W_{t,i}$'s terms of the $\tau$-series of
$W$.

\item[(3).] The initial condition $x(0)$ of the exact system being
         given, truncate to the order $m$ the $K_0$-transformation (at
time $t=0$) taken at $x(0)$, $\mathcal{T}_{K_0}(1)\cdot
{K_0}(x(0))$, where $K$ is loosely speaking the generator of the
family of inverse transformations, and take this truncation plus
$x(0)$ as initial condition for the $n^{th}$ averaged system.

\item[(4).] Compute numerically the solution $\overline{x}^{(n)}$ of
      the $n^{th}$ averaged
      system, and add to $\overline{x}^{(n)}(t)$ the truncation to the order $m$ of the series
      $\mathcal{T}_{W_t}(1)\cdot {W_t}(\overline{x}^{(n)}(t))$, which gives
finally the {\em $m^{th}$ approximation of the $n^{th}$ type} of the
exact solution at time $t$.
\end{itemize}

\section{Application to a problem in atmospheric
chemistry}\label{numtest} The original problem that motivated this
work was to examine models of diurnal forcing in atmospheric
dynamics and chemistry on long time-scales \cite{det}. The
day-to-night changes in the radiative heating and cooling of the
planetary boundary layer -- the lowest part of the atmosphere (1-2
km) -- are very large. Still, one is often only interested in the
slow, season-to-season or even year-to-year changes in the way this
lower layer interacts with the underlying surface (land or ocean),
on the one hand, and the free atmosphere above, on the other
\cite{det}. The diurnally averaged model we derived here provides
insight into a similar problem, that of the basic chemistry of slow
changes, from one day or week to the next, of a highly simplified
system of photochemically active trace gases in the troposphere
(i.e., the lower 10 km of the atmosphere). The system of two coupled
ODEs we consider in this paper governs the concentration of the
chemical species CO (carbon monoxide) and O3 (ozone) \cite{hem},
\begin{equation}\label{chemicalsystem}
M2\left\{ \begin{array}{l}
\frac{dx_{1}}{dt}=S_{1}(t)-Z_{1}(t)x_{1}x_{2}\\ \noindent
\frac{dx_{2}}{dt}=-S_{2}(t)+Z_{1}(t)x_{2}x_{1}-(1+Z_{1}(t))x_{2}+(Z_{2}(t)+
S_{2}(t))\frac{1}{x_{2}}, \end{array} \right.
\end{equation}
where $t\to x(t)=(x_{1}(t),x_{2}(t))^{T}=([CO](t),[O_{3}](t))^{T}$.
This system belongs to the general class of ODEs systems given by
(\ref{generalsystem}).

In this system, the diurnal forcing is through the functions $S_{i}$
and $Z_{i}$ (see Figure \ref{coeffs} below for a typical example),
which can have rather complicated shapes.  The system
(\ref{chemicalsystem}) is a very simple model for air pollution in
an urban environment.  Changes in the chemistry at the day-night
transitions and those in the emission of CO and O3 in a city during
a day lead to $S_{i}$ and $Z_{i}$ which are only piece-wise smooth.

%\graphicspath{{DCDS-Format/}{eps/}}
\begin{figure}
  \centering
\includegraphics*[width=8.5cm]{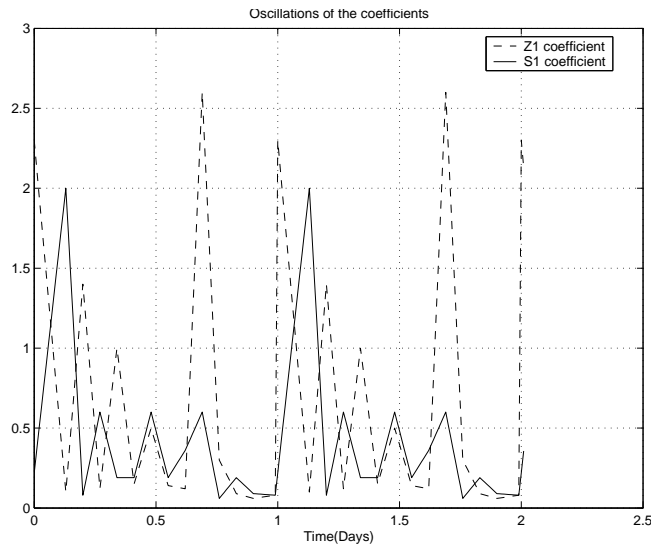}
\caption{\label{coeffs} The figure depicts forcing induced by $S_i$
and $Z_i$.}
 \vspace{0.1cm}
 \center{}
\end{figure}

While it is not difficult to numerically integrate the system as it
is by standard methods, the forcing carried by functions $S_{i}$ and
$Z_{i}$ can have high frequencies and therefore require a time step
too small compared to the total time of simulation, $T_{max}$ (say,
several tens of years). This assumes that the goal is to obtain
sufficiently smooth and therefore realistic numerical solutions.

We discuss in this section how well the numerical solutions of the
averaged systems, at orders one and two, approximate those of the
original one. An important part of the numerical work for this
particular example is to carry out  the transformations between the
solutions of the original and averaged systems developed in \S
\ref{LTS}. This includes the transformation of initial values of \S
\ref{seccondinit}. Overall, the numerical simulations indicate that
the correction performed by Lie transforms to the first order of the
first averaged system $\overline{M2}$ or the second averaged system
$\overline{M2}^{(2)}$ obtained by applying results of \S
\ref{higherorder}, provides a good approximation to the original
one. Furthermore, regularity  of the averaged systems in time allows
one to integrate them by simple methods, such as an Euler or a
Runge--Kutta of order 4 (RK4 in the sequel).

Last but not  least, we can also use larger step sizes for the
averaged systems than for the original one, which is the key to
speed up simulations of long-term dynamical phenomena.

\subsection{First and second averaged systems
analysis}\label{Analysis}

Let $\Omega:=\mathbb{R}^{2}\backslash \{x_2=0\}$.  Then the vector
field $Y$ associated to $M2$ belongs to
$\mathcal{P}^{\infty}(\Omega)$ and
$\overline{Y}\in\mathcal{C}^{\infty}(\Omega)$. Assume that there
exists, for each $t\geq 0$, a diffeomorphism $\phi_{1,t}\in
\mathcal{P}^{\infty}_{d}(\widetilde{\Omega})$, generated by
$\widetilde{G}_t=(1,G_t)^{T}$, such that $(\phi_{1,t})^{\ast}
\widetilde{Y}=\widetilde{Y}_{ave}$.

In order to demonstrate the numerical efficiency of the averaged
systems, we present only an approximation of the family of
diffeomorphisms $(\phi_{1,t})_{t\in\mathbb{R}^{+}}$ up to order one.

For that, we compute the first corrector $G_{0,t}^{[0]}=G_{0,t}$
(see Definition 2.8) given by integration of (\ref{f21}). Define
$G_{0,1}(t,\xi )$ (resp. $G_{0,2}(t,\xi ))$ as the first (resp.
second) component of $G_{0}(t,\xi )$. Then, from
(\ref{chemicalsystem}),
\begin{align}  \label{f23}
G_{0,1}(t,\xi )& = \int_{0}^{t}\delta S_{1}(s)ds-\xi _{1}.\xi
_{2}\int_{0}^{t}\delta Z_{1}(s)ds+C_{0,1}(\xi _{1},\xi _{2}),
\\  \label{f24}
G_{0,2}(t,\xi )& =\left\{
\begin{array}{l}
(\frac{1}{\xi _{2}}-1).\int_{0}^{t}\delta S_{2}(s)ds+(\xi _{1}-1)\xi
_{2}.\int_{0}^{t}\delta Z_{1}(s)ds \\
+\frac{1}{\xi _{2}}\int_{0}^{t}\delta Z_{2}(s)ds+C_{0,2}(\xi
_{1},\xi _{2})
\end{array}
\right\},
\end{align}
where, for $i\in\{1,2\}$,
\begin{equation}\label{magncoeffs}
\delta S_{i}(t)=S_{i}(t)-\overline{S_{i}} \mbox{ and } \delta
Z_{i}(t)=Z_{i}(t)-\overline{Z_{i}}, \mbox{ for all } \quad t \in
\mathbb{R^{+}}.
\end{equation}
The constant parts $C_{0,1}$ and $C_{0,2}$  of (\ref{f23}) and
(\ref{f24}) are given by assuming that for all $\xi\in\Omega$,
$\int_{0}^{T}\int_{0}^{t}G_0(\xi,s)dsdt=0$, which allows us to avoid
``secular" terms at the first order, namely
\begin{equation}
C_{0,1}(\xi
_{1},\xi_{2})=-\int\limits_{0}^{T}\int\limits_{0}^{t}\delta
S_{1}(s)dsdt+\xi
_{1}\xi_{2}\int\limits_{0}^{T}\int\limits_{0}^{t}\delta Z_{1}(s)dsdt
\label{f(38)},
\end{equation}
and from  (\ref{f24}):
\begin{equation}  \label{f(39)}
C_{0,2}(\xi_{1,}\xi_{2})=\left\{
\begin{array}{l}-(\frac{1}{\xi_{2}}-1)\int\limits_{0}^{T}\int\limits_{0}^{t}\delta
S_{2}(s)dsdt-(\xi_{1}-1)\xi
_{2}\int\limits_{0}^{T}\int\limits_{0}^{t}\delta Z_{1}(s)dsdt
\\-\frac{1}{\xi _{2}}\int\limits_{0}^{T}\int\limits_{0}^{t}\delta
Z_{2}(s)dsdt
\end{array}
\right\}.
\end{equation}

Now, let $x(0)$ be the initial condition of a solution $x$ in system
$M2$. In order to compute the first approximation of the first type,
following the procedure described at the end of  \S
\ref{seccondinit}, we have to truncate to order one the
$K_0$-transformation at time $t=0$ of $K_0$ taken at $x(0)$, i.e.
$\mathcal{T}_{K_0}(1)\cdot {K_0}(x(0))$, where $K$ is, loosely
speaking, the generator of the family of inverse transformations.
Thereby, we get  a first approximation of the initial condition
$\overline{x}(0)$ of the solution $\overline{x}$ which is
$\mathcal{P}^{\infty}_{diff}$-correspondent to $x$ by the family
$(\phi_{1,t})_{t\in \mathbb{R}^{+}}$.

So we have by Proposition \ref{inverserepre}, end of \S
\ref{seccondinit} and Theorem 7.1, that the first approximation of
the first type $\overline{x}^{(1,1)}(0)$ of the initial condition
$\overline{x}(0)$ is
\begin{equation}\label{1stappcondinit}
\overline{x}^{(1,1)}(0)=x(0)-G_0(0,x(0)),
\end{equation}
with $G_0(0,x(0))$ determined by the preceding  constants evaluated
at $x(0)$. Note that the first index of the exponent $(1,1)$ in
(\ref{1stappcondinit}) indicates that we deal with the first type of
approximation and note that the second superscript indicates that we
make a truncation up to order one in the series
$\mathcal{T}_{K_0}(1)\cdot {K_0}(x(0))$. An identical convention
will be used in the following when we work with approximations of
higher-order type.

Next we numerically  compute  the solution $z$ of $\overline{M2}$
through $\overline{x}^{(1,1)}(0)$ and then we define the {\em first
pullback} $x^{(1)}$ by truncating to ``order one" the expression
\\ $z(t)+ \mathcal{T}_{G_t}(1)\cdot{G_t}(z(t))$, which gives:
\begin{equation}\label{1stpub}
x^{(1)}(t):=z(t)+G_0(t,z(t)), \mbox{ for all } t\in\mathbb{R}^{+}.
\end{equation}

We describe now the construction of the {\em second pullback}. By
Proposition \ref{nthaveraging} with $n=2$, the second averaged
system $\overline{M2}^{(2)}$ corresponds to the vector field
$\overline{Y}^{(2)}=\overline{Y}+\frac{1}{2}Y_0^{(2)}$ where
$Y_0^{(2)}$ is  determined according to (\ref{eltn}) with $m=2$.

Both direct computation by hand and symbolic manipulation software
for $ Y_{0}^{(2)}=\left[ Y_{0,1}^{(2)}, Y_{0,2}^{(2)} \right]^{T} $,
yield for system (\ref{chemicalsystem}) the following expression as
function of $\xi _{1}$ and $\xi _{2}$, namely,
$$
Y_{0,1}^{(2)}=\frac{1}{T}\int\limits_{0}^{T}\left\{
\begin{array}{l}
(\int\limits_{0}^{t}\delta
S_{1}(s)ds-\int\limits_{0}^{T}\int\limits_{0}^{t}\delta
S_{1}(s)dsdt).(-\xi _{2}.\sigma Z_{1}(t)) \\
+(\int\limits_{0}^{t}\delta
S_{2}(s)ds-\int\limits_{0}^{T}\int\limits_{0}^{t}\delta
S_{2}(s)dsdt).[\xi _{1}.(1-\frac{1}{\xi _{2}} ).\sigma
Z_{1}(t)] \\
+(\int\limits_{0}^{t}\delta
Z_{1}(s)ds-\int\limits_{0}^{T}\int\limits_{0}^{t}\delta
Z_{1}(s)dsdt)\left\{
\begin{array}{l}
[\xi _{2}.\sigma S_{1}(t)\\
+\xi _{1}.(\frac{1}{\xi _{2}}-1)\sigma S_{2}(t)\\
+ \frac{\xi _{1}}{_{\xi _{2}}}\sigma Z_{2}(t)\\
-2.\xi_{1}\xi _{2}]
\end{array}
\right\} \\
+( \int\limits_{0}^{t}\delta
Z_{2}(s)ds-\int\limits_{0}^{T}\int\limits_{0}^{t}\delta
Z_{2}(s)dsdt).(-\frac{\xi _{1}}{_{\xi _{2}}}.\sigma Z_{1}(t))
\end{array}
\right\} dt,
$$
and:
$$
Y_{0,2}^{(2)}=\frac{1}{T}\int\limits_{0}^{T}\left\{
\begin{array}{l}
(\int\limits_{0}^{t}\delta
S_{1}(s)ds-\int\limits_{0}^{T}\int_{0}^{t}\delta S_{1}dsdt).(\xi
_{2}.\sigma Z_{1}(t)) \\
+(\int\limits_{0}^{t}\delta
S_{2}(s)ds-\int\limits_{0}^{T}\int\limits_{0}^{t}\delta
S_{2}dsdt).\left\{
\begin{array}{l}
[\alpha\sigma Z_{1}(t)\\
+ \frac{1}{(\xi _{2})^{2}}\sigma Z_{2}(t)\\
+2(1-\frac{2}{\xi _{2}})\\
+ \beta \sigma S_{2}(t)]
\end{array}
\right\} \\
+(\int\limits_{0}^{t}\delta
Z_{1}(s)ds-\int\limits_{0}^{T}\int\limits_{0}^{t}\delta
Z_{1}dsdt).\left\{
\begin{array}{l}
[-\xi _{2}\sigma S_{1}(t)\\
+\gamma\sigma S_{2}(t)\\
-(\xi _{1}-1)\frac{2}{\xi _{2}}\sigma Z_{2}(t)]
\end{array}
\right\} \\
+(\int\limits_{0}^{t}\delta
Z_{2}(s)ds-\int\limits_{0}^{T}\int\limits_{0}^{t}\delta Z_{2}dsdt).
\left\{
\begin{array}{l}
[-\frac{1}{(\xi _{2})^{2}}\sigma S_{2}(t)\\
+\frac{2}{\xi _{2}}(\xi _{1}-1)\sigma Z_{1}(t)\\
-\frac{4}{\xi_{2}}]
\end{array}
\right\}
\end{array}
\right\} dt,
$$
where $\alpha$ (resp. $\beta$, $\gamma$) are defined by
$\alpha=2\frac{\xi_{1}}{\xi_{2}}-\frac{2}{\xi_{2}}-\xi_{1}+1$ (resp.
$\beta=\frac{2}{(\xi_{2})^{2}}(\frac{1}{\xi_{2}}-1)$, $\gamma=(\xi
_{1}-1)(1-\frac{2}{\xi _{2}})$), $\sigma
Z_{i}(t)=Z_{i}(t)+\overline{Z_{i}}$  and $\sigma
S_{i}(t)=S_{i}(t)+\overline{S_{i}}$ for all $t$. The terms $\delta
Z_i$ and $\delta S_i$ are defined in (\ref{magncoeffs}).

Denote by $\overline{x}^{(2)}(0)$ the initial condition for
$\overline{M2}^{(2)}$ through which the solution
$\overline{x}^{(2)}$ is $\mathcal{P}^{\infty}_{diff}$-correspondent
to $x$. We take as approximation of $\overline{x}^{(2)}(0)$,
following the preceding procedure and notations concerning the first
averaged system and taking into account the invariance of the first
corrector (cf (\ref{firsteq})), the vector $\overline{x}^{(2,1)}(0)$
is given by
\begin{equation}\label{2ndappinit}
\overline{x}^{(2,1)}(0)=x(0)-G_0(0,x(0)),
\end{equation}
i.e. $\overline{x}^{(1,1)}(0)$ defined in (\ref{1stappcondinit}).

Then we numerically compute the solution $v$ of
$\overline{M2}^{(2)}$ based on $\overline{x}^{(2,1)}(0)$, and define
the {\em second pullback} $x^{(2)}$ as the first approximation of
the second type, i.e.:
\begin{equation}\label{2ndpub}
x^{(2)}(t):=v(t)+G_0(t,v(t)), \mbox{ for all } t\in \mathbb{R}^{+}.
 \end{equation}

The choice of computing only the first approximation, based on the
first and second averaged systems, allows us to compare the
approximations of the original system $M2$ by these averaged forms,
the corrector  being the same in each case, but applied to different
solutions. This will be discussed in \S  \ref{Accu}.

The numerical experiments we present here correspond to choices of
the forcing functions ($S_{i},Z_{i}$) which produce ``broad"
oscillations in the solutions of the original system. These are not
very realistic but our goal here is to test how well the method
performs for ``large perturbations".
 The magnitude of perturbations is shown in Figure \ref{coeffs}. In all the
numerical simulations, $ Z_{1} $ and $ S_{1} $  were oscillatory but
we used the constants $ Z_{2}=5.0\times 10 ^{-2} $ and $ S_{2}=12.0
\times 10^{-2} $.

\begin{figure}
%[b]{.85\linewidth}
\centering
\includegraphics*[width=9.5cm]{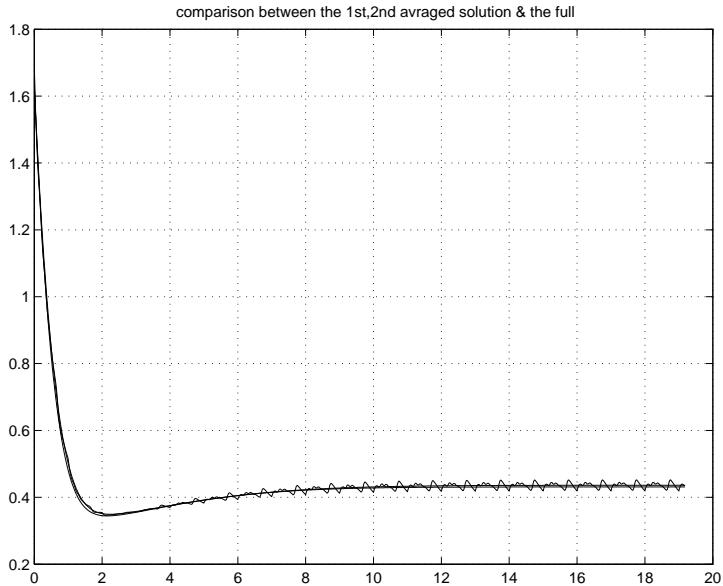}
  % \centerline{\epsfig{figure=\figdir /chapmodel/fuzzyobj.eps,width=\linewidth}}
\caption{\label{ave12-2nd} The figure shows the comparison between
the full solution and the $1^{st}$ and the $2^{nd}$ averaged
solutions, on the $2^{nd}$ component, for the total time of
simulation.}
   %\vspace{0.1cm}
% \center{}\medskip
%\end{minipage}
\end{figure}

\subsection{Gain in numerical efficiency}\label{GCPU}

The original system is integrated by standard methods (Euler,
Runge--Kutta,...) with step size $\delta t$, which has to be small
enough to resolve the oscillations induced by the forcing terms
$Z_i$ and $S_i$. The averaged system $\overline{M2}$ or
$\overline{M2}^{(2)}$ is integrated with step size $\triangle t$,
which is typically ten times larger than $\delta t$.

A key point in numerical efficiency is the regularity of the
solutions of the averaged system. Indeed, if the solutions of the
averaged system are smooth enough in the sense that they do not show
multiple oscillations over one forcing period,  which is one day in
our case, we can integrate the averaged system with a large step
size $\triangle t$, without loss of regularity. This fact for system
(\ref{chemicalsystem}) is pointed out in Figure \ref{ave12-2nd}, for
instance, where the solutions of the averaged systems
$\overline{M2}$ and $\overline{M2}^{(2)}$, on the second component,
are those which do not show multiple oscillations over one day. This
fact is well known in the $\epsilon$-dependent case where the drift
described by the averaged system can be integrated with a step size
chosen to be $\frac{1}{\epsilon}$ times larger than for the
non-averaged system \cite{a}.

In order to investigate how well the solutions of the averaged
system corrected up to order one approximate that of the original
one, we did the following:
\begin{enumerate}
\item[\textrm{(i)}] For a given initial condition for the original
system, $x(0)$, we compute the $1^{st}$-approximation of the
$1^{st}$-type $\overline{x}^{(1,1)}(0)$ and  the
$1^{st}$-approximation of the $2^{nd}$-type
$\overline{x}^{(2,1)}(0)$ given by (\ref{1stappcondinit}) and
(\ref{2ndappinit}), of the initial conditions $\overline{x}(0)$ and
$\overline{x}^{(2)}(0)$, respectively.

 \item[\textrm{(ii)}] Starting from these initial conditions, we compute
$\overline{x}$  and $\overline{x}^{(2)}$ by integrating
$\overline{M2}$ and $\overline{M2}^{(2)}$ with $\triangle t$ as step
size, by using a standard integrator. This costs less than
integrating $M2$ with a small enough step size that resolves the
oscillating forcing terms. Here it is sufficient to use the Euler
method.

\item[\textrm{(iii)}]  We do a simple linear interpolation on
$\overline{x}$ and $\overline{x}^{(2)}$ to obtain the solution on
the temporal grid defined by $\delta t$.

\item[\textrm{(iv)}]  According to  (\ref{1stpub}) (resp.
(\ref{2ndpub})) and the fact that the correctors are only composed
of integrals, (\ref{f23}) and (\ref{f24}) are used for computing, on
the grid defined by $\delta t$,  the first and the second pullback
based on solutions obtained in (ii).
\end{enumerate}

The numerical tests indicate that the solution of $M2$, obtained by
the procedure described through the items (i) to (iv), has accuracy
which is close to that obtained by integrating $M2$ itself by
classical methods, such as a RK4 scheme.  The CPU time linked with
our procedure depends essentially on the one defined for solving the
averaged systems $\overline{M2}$ or $\overline{M2}^{(2)}$ with step
size $\triangle t$, and on the one for the integral correction
procedure.

We observed that the use of the first pullback allows to obtain a
gain in CPU time nearly equal to $75$ percent when we use a large
step size for $\triangle t$, while retaining good accuracy and
regularity of the solutions,  as shown in the following subsection
where the tests and the related figures are for $\delta t =0.01$ and
$\triangle t= 0.1$.

\subsection{Accuracy achieved by the first and second
pullback}\label{Accu} One of the main reasons for computing the
second averaged system $\overline{M2}^{(2)}$ were  to obtain a
better approximation. In this section we demonstrate numerically
this fact for $\delta t=0.01$ and $\triangle t= 0.1$. Note that the
tests performed by RK4 on $M2$ with $\triangle t$ give no smooth
solutions (not shown), whereas, as shown in Figures \ref{osc12-1st}
and \ref{osc12-2nd}, this is not the case for solutions obtained by
our method described in (i)-(iv) of  \S \ref{GCPU}.

\begin{figure}
\begin{minipage}[b]{.85\linewidth}
  \centering
\includegraphics*[width=9.5cm]{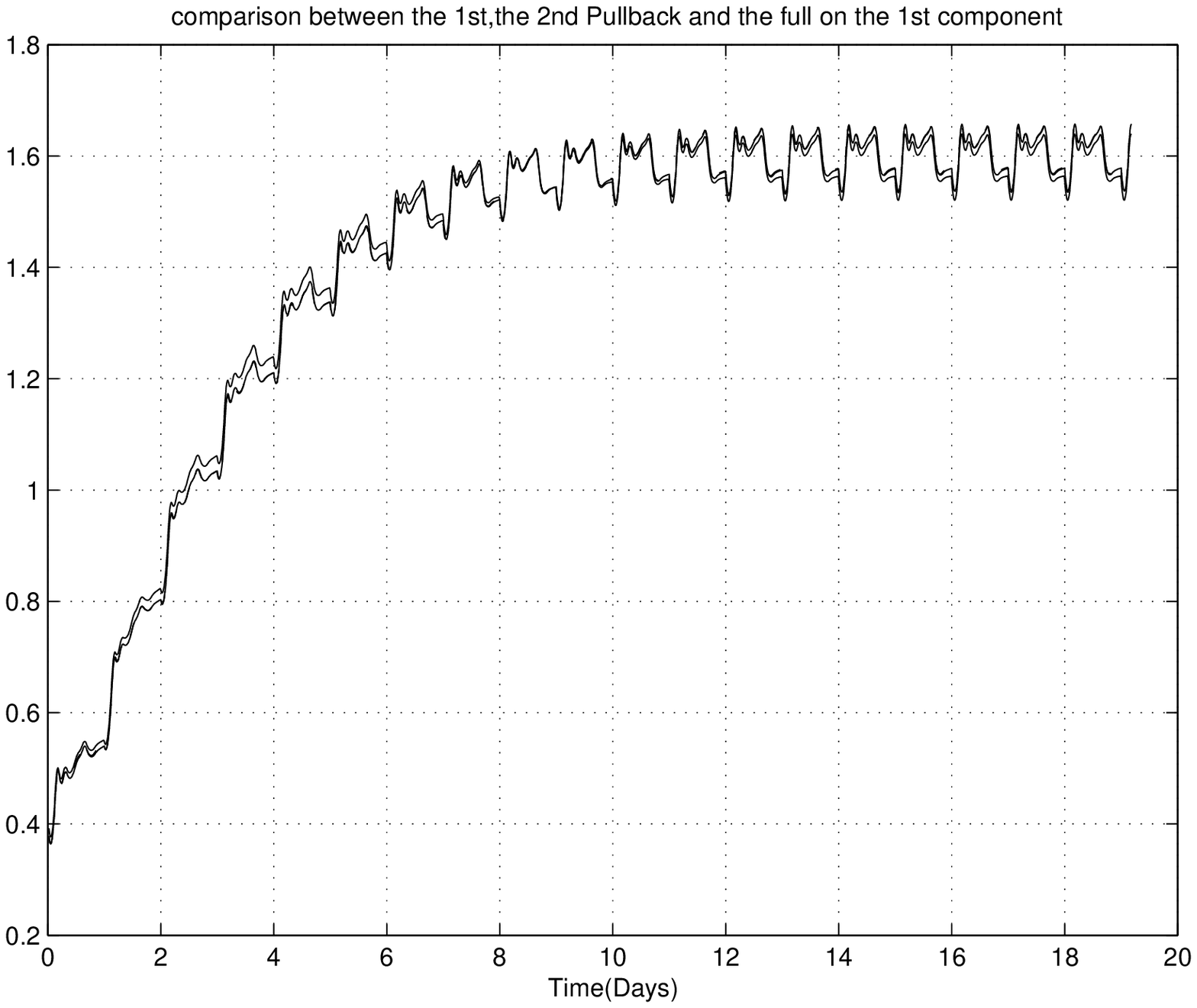}
\caption{\label{osc12-1st}
 The figure shows the comparison between the full solution and the $1^{st}$ and the $2^{nd}$ pullback solutions,
on the $1^{st}$ component, for the total time of simulation.}
  % \centerline{\fbox{\epsfig{figure=\figdir /chapmodel/coin.eps,width=.8\linewidth,height=.8\linewidth}}}
 \vspace{0.1cm}
 \center{}\medskip
\end{minipage}
\hfill
\begin{minipage}[b]{.85\linewidth}
\centering
\includegraphics*[width=9.5cm]{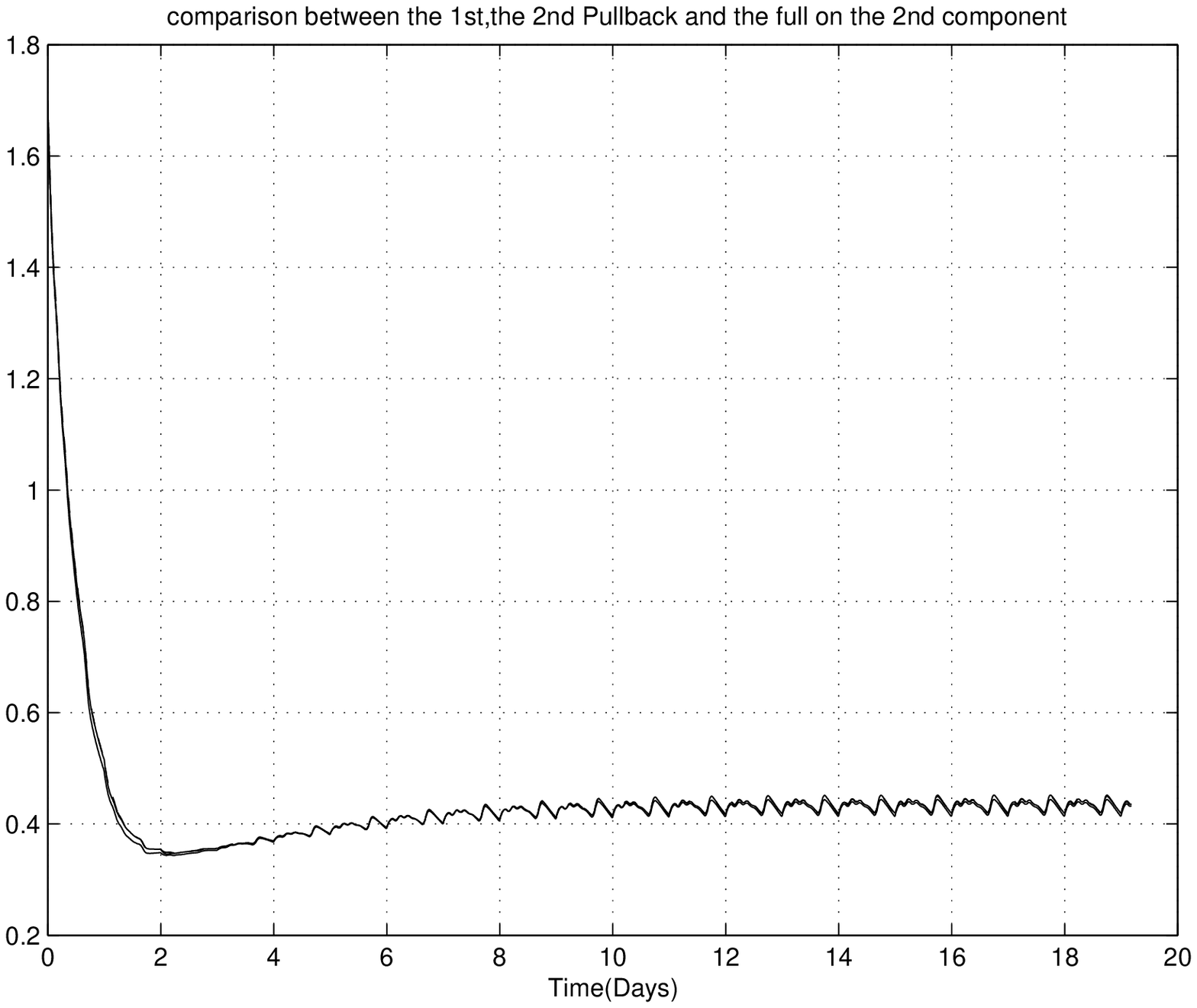}
%\epsfig{file=mit1.eps, width=6.5cm }
  % \centerline{\epsfig{figure=\figdir /chapmodel/fuzzyobj.eps,width=\linewidth}}
\caption{\label{osc12-2nd} The figure shows the comparison between
the full solution and the $1^{st}$ and the $2^{nd}$ pullback
solutions, on the $2^{nd}$ component, for the total time of
simulation.}
   %\vspace{0.1cm}
% \center{}\medskip
\end{minipage}
\end{figure}

Based on the method of approximation used here, in many cases, the
second pullback is better than the first, and gives a solution which
is close to the one obtained by RK4.  In Figures \ref{osc12-1st} and
\ref{osc12-2nd}, we observe that we have coincidence of the
$1^{st}$-, $2^{nd}$-pullback and the full solution for the global
trend in each component. A more accurate analysis shows that the
$2^{nd}$-pullback is a better approximation than the first as shown
in Figure \ref{zoomosc-1st} and Figure \ref{zoomosc-2nd} where the
$2^{nd}$-pullback is represented by dash, the $1^{st}$-pullback by
dots, and the solution computed with RK4 by solid line. This fact is
made more evident by the analysis of the error in the supremum norm
as shown in Figure \ref{err-1st}, where the error produced by the
$2^{nd}$-pullback is under $6\times 10^{-3}$ and in Figure
\ref{err-2nd} where the error with the $2^{nd}$-pullback is
represented by dash.

Nevertheless, we can easily imagine that for a system of dimension
larger than two the CPU time required to calculate the
$2^{nd}$-pullback would increase with algebraic complexity of the
$2^{nd}$-averaged system (cf. \S \ref{Analysis}), but as we shall
see in \S \ref{period} the $2^{nd}$-pullback can be used for the
problem of determination of periodic solutions in a general
framework.

\begin{figure}
\begin{minipage}[b]{.85\linewidth}
  \centering
\includegraphics*[width=9.5cm]{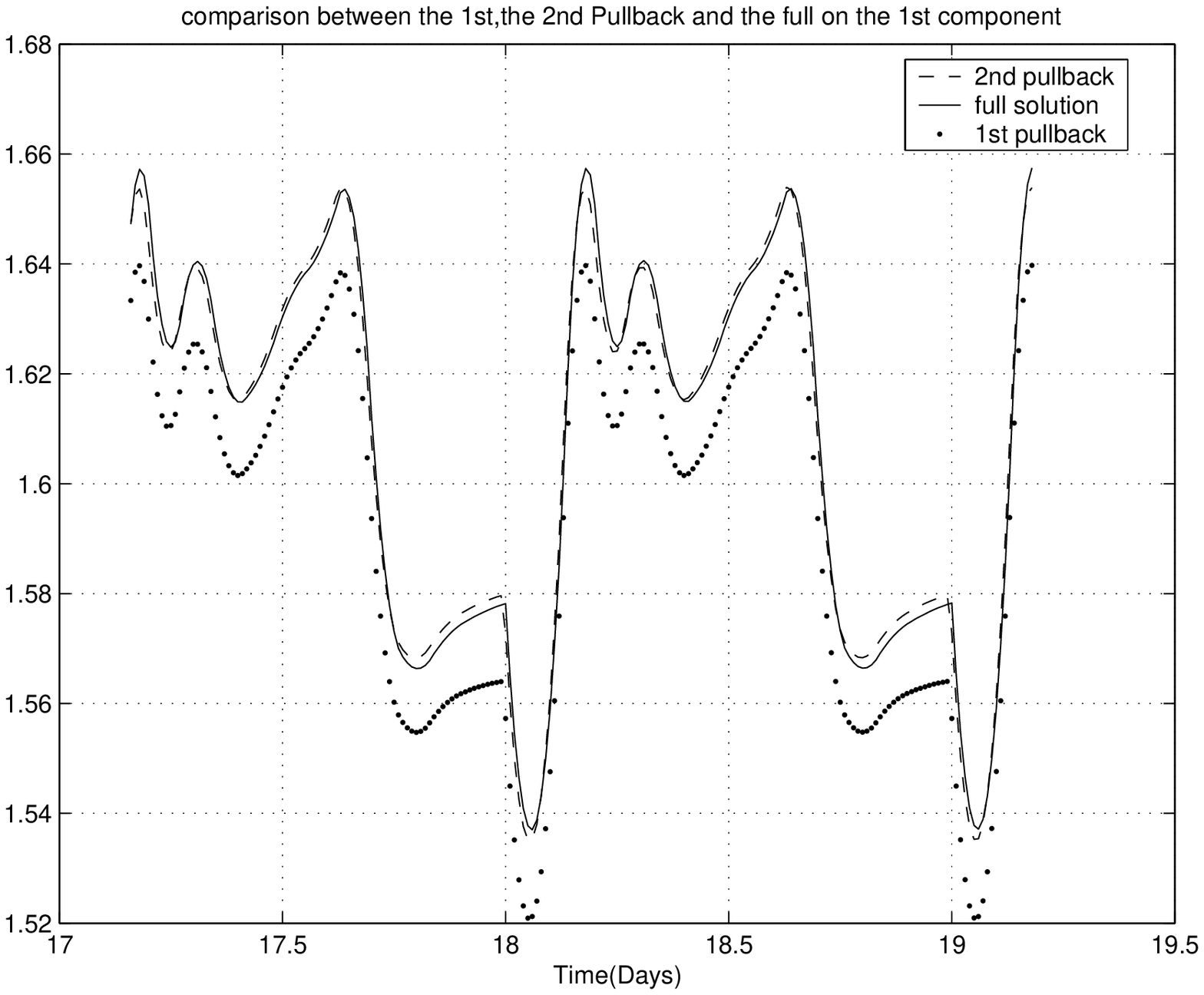}
\caption{\label{zoomosc-1st}The figure shows the comparison between
full solution and the $1^{st}$ and the $2^{nd}$ pullback solutions,
on the $2^{nd}$ component, for the periodic regime.}
  % \centerline{\fbox{\epsfig{figure=\figdir /chapmodel/coin.eps,width=.8\linewidth,height=.8\linewidth}}}
 \vspace{0.1cm}
 \center{}\medskip
\end{minipage}
\hfill
\begin{minipage}[b]{.85\linewidth}
\centering
\includegraphics*[width=9.5cm]{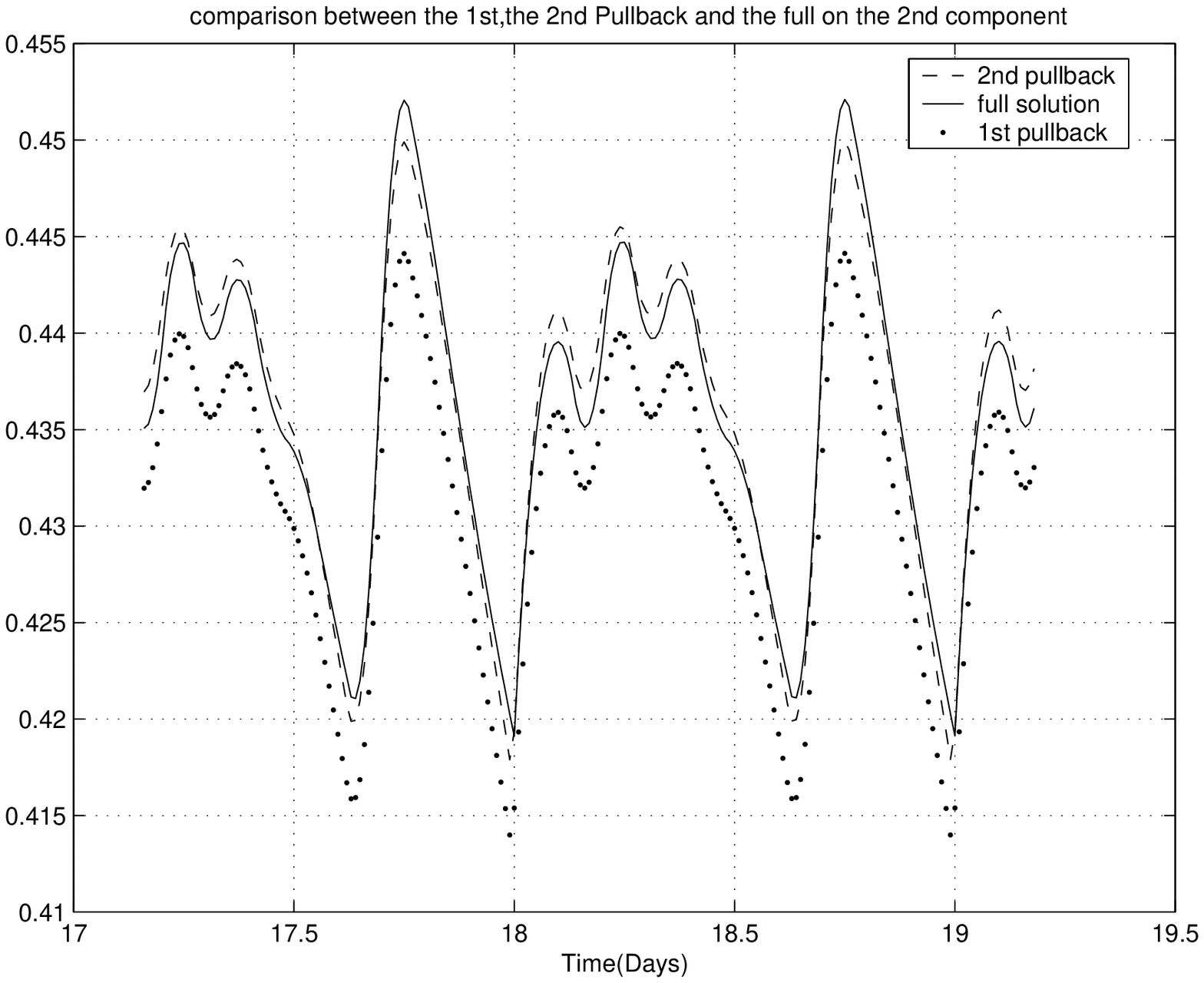}
%\epsfig{file=miti4.eps, width=6.5cm }
  % \centerline{\epsfig{figure=\figdir /chapmodel/fuzzyobj.eps,width=\linewidth}}
\caption{\label{zoomosc-2nd} The figure shows the comparison between
full solution and the $1^{st}$ and the $2^{nd}$ pullback solutions,
on the $2^{nd}$ component, for the periodic regime.}
  % \vspace{0.1cm}
% \center{}\medskip
\end{minipage}
\end{figure}

The tests presented in this study give satisfactory results, an
accuracy of order $1 \times 10^{-3}$ for the second pullback.
Moreover, we have to note that this accuracy is obtained with
oscillations actually ``far" from their averages (see Figure
\ref{coeffs} again) producing oscillations, of magnitude about
$1\times 10^{-1}$, on the full solution obtained by RK4, with a
significant gain in CPU time compared to the standard method.

Therefore we can conclude that the Lie transforms averaging method
developed here gives a rigorous setting for constructing the
correctors and $\overline{M2}^{(2)}$ to provide numerical
approximations even for relatively large perturbations in time
induced by the forcing terms. Furthermore, according to the tests
performed in this work, the $2^{nd}$-averaged system analysis seems
to be more relevant than the first.

\begin{figure}[!hbtp]
\begin{minipage}[b]{.85\linewidth}
  \centering
\includegraphics*[width=9.5cm]{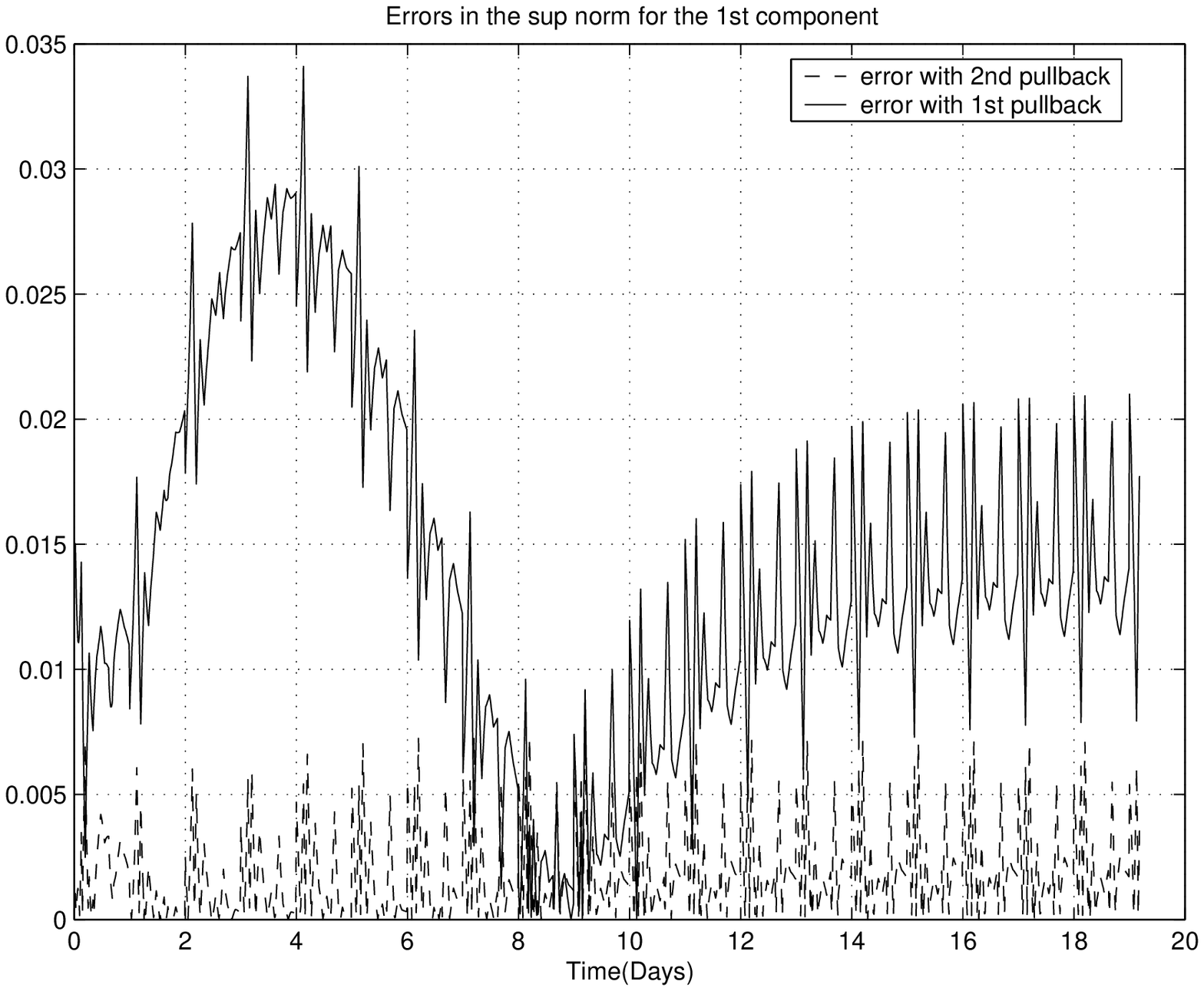}
\caption{\label{err-1st} The figure shows the error, in the supremum
norm, between the full solution and the $1^{st}$ and the $2^{nd}$
pullback solutions, on the $1^{st}$ component, for the total time of
simulation.}
  % \centerline{\fbox{\epsfig{figure=\figdir /chapmodel/coin.eps,width=.8\linewidth,height=.8\linewidth}}}
 \vspace{0.1cm}
 \center{}\medskip
\end{minipage}
\hfill
\begin{minipage}[b]{.85\linewidth}
\centering
\includegraphics*[width=9.5cm]{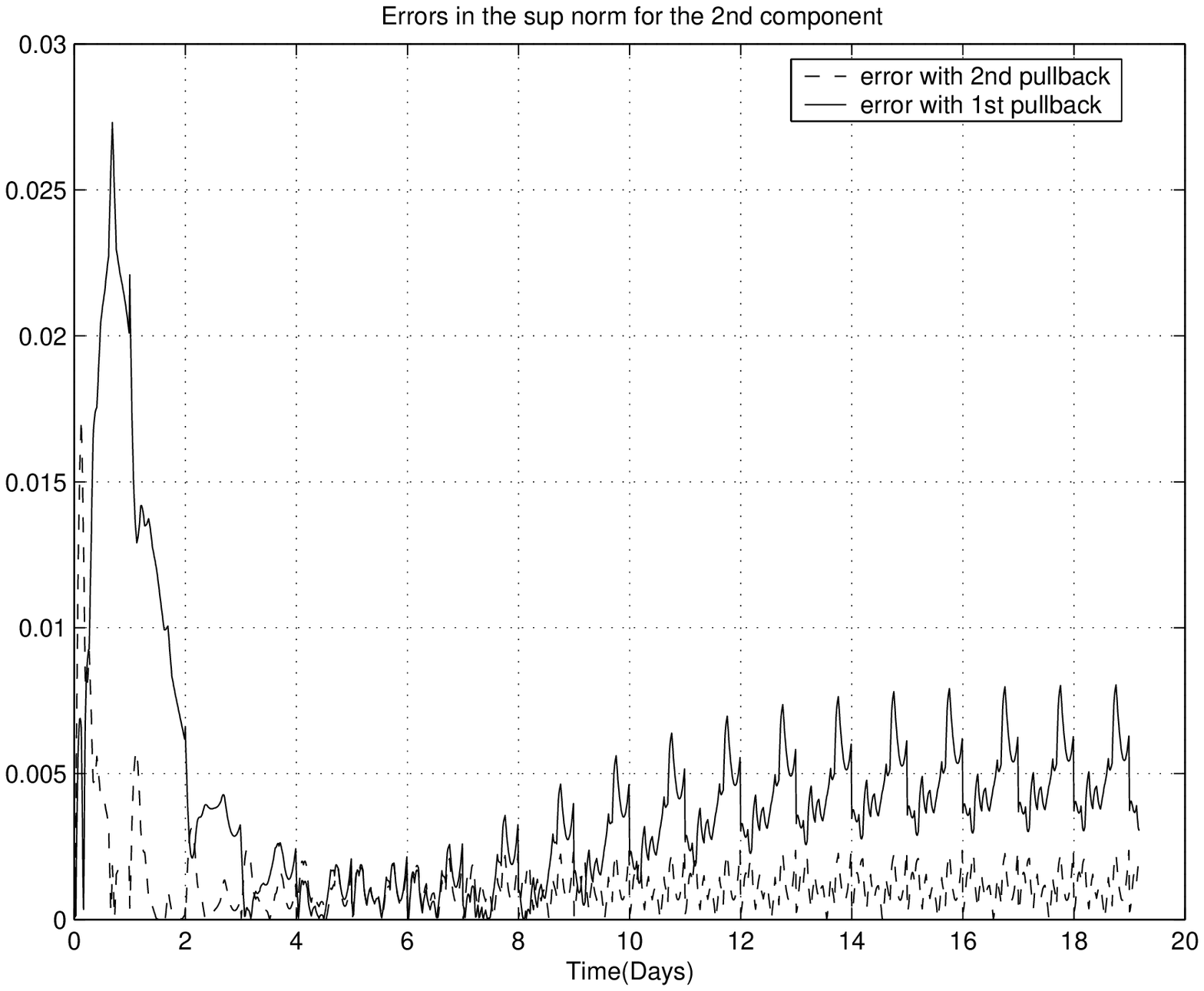}
  % \centerline{\epsfig{figure=\figdir /chapmodel/fuzzyobj.eps,width=\linewidth}}
\caption{\label{err-2nd} The figure shows the error, in the supremum
norm, between the full solution and the $1^{st}$ and the $2^{nd}$
pullback solutions, on the $2^{nd}$ component, for the total time of
simulation.}
\end{minipage}
\end{figure}
\section{Computing periodic solutions by higher-order averaged systems
analysis: some considerations}\label{period}

In the classical approach, i.e. in the $\epsilon$-dependent case, in
order to find periodic solutions the theory of averaging can be a
useful tool as shown, for instance, by theorem 4.1.1 (ii) of J.
Guckenheimer \& P. Holmes \cite{gh} for the first averaged system,
Hartano and  A.H.P. van der Burgh in \cite{har}, or A. Buic\u{a} \&
J. Llibre in \cite{bl} with relaxed assumptions, for the first,
second and third averaged systems, based on Brouwer degree theory.
All these results are linked to a maximal size of perturbation
$\epsilon_0$ under which the existence of a hyperbolic fixed point
$p_0$ of the first averaged system gives the existence of a unique
periodic solution of the original system, revolving around $p_0$.

As was explained in \S \ref{heuristic}, the analysis performed in
this paper permits to get such results for higher-order averaging
analysis based on explicit formulae (Proposition \ref{nthaveraging}
for averaging and Theorem \ref{diffeorepre} for corrections) by
revising the proof of classical results for first and second order
averaged systems, given in the literature (e.g. \cite{gh}). For
instance, if the $k$ ($k\geq 1$) first averaged systems vanish
identically, then a proof of existence and uniqueness of a
$T$-periodic solution in a $\epsilon$-neighborhood of a hyperbolic
fixed point of the $(k+1)^{th}$-averaged system can be given on the
basis of principles given in the proof of theorem 4.1.1 of
\cite{gh}, essentially by showing that the Poincar\'{e} maps of
nonautonomous and autonomous systems are $\epsilon$-closed.

In the light of such considerations, a question arises naturally: do
there exist such results for the $\epsilon$-independent case?

A preliminary analysis can be done on time $T$-maps. Indeed, we can
formulate the following lemma:

\vspace{1ex}

\noindent{\bf Lemma 6.1.}
%\begin{Lem}
\label{conjugPoinc}{\em Let $Y\in\mathcal{P}^{r}(\Omega)$, and
$Z\in\mathcal{C}^{r}(\Omega), (1\leq r\leq \infty)$. Then Y and Z
are $\mathcal{P}^{r}_{diff}$-equivalent if and only if their time
$T$-maps are $\mathcal{C}^{r}$-conjugate.}
%\end{Lem}
%\refstepcounter{theorem}
 \vspace{2ex}

This lemma gives thereby a way to study sufficient conditions for
the existence of solutions of the nonlinear functional equation
(\ref{bigproblem}), which will be investigated in a forthcoming
paper. Note that the existence of time $T$-maps is realized under
the assumption $(\lambda)$.

In practice, if the conjugacy of time $T$-maps is achieved, we can
localize a $T$-periodic orbit of the nonautonomous system $Y$.
Indeed, suppose that the time $T$-map associated with an averaged
system $Z$ has a fixed point. Then by conjugacy, this fact still
holds for the time $T$-map associated with the nonautonomous system.
Making use of the proof of Lemma 6.1, we can observe that if
$(\pi\circ\phi_{1,t}\circ\mathcal{I}_t)_{t\in\mathbb{R}^{+}}$
denotes the solution of $(\ref{bigproblem})$ obtained by Lie
transforms, then $\pi\circ\phi_{1,o}^{-1}\circ\mathcal{I}_0$ is the
diffeomorphism realizing the conjugacy between time $T$-maps
associated with $Y$ and $Z$. Therefore Proposition
\ref{inverserepre} and Theorem \ref{diffeorepre} allow us to compute
an approximation of this diffeomorphism, which can be applied to a
fixed point $\eta$ associated with the system $Z$, leading to an
approximation $\xi_0$ of an initial datum lying on the $T$-periodic
orbit of the nonautonomous system which is  in
$\mathcal{P}_{diff}^{\infty}$-correspondence with $\eta$. Thus, by a
standard integrator based at $\xi_0$, we can compute an
approximation of this $T$-periodic solution. Such a method can be a
useful tool for localizing $T$-periodic solutions of $T$-periodic
nonautonomous dissipative systems, which are known to exist, in this
case,  as shown by classical results on the topic (cf. \cite [p.
235]{kra}).

This procedure is relevant for our system (\ref{chemicalsystem}), as
proved by elementary analysis. Indeed, there exists a hyperbolic
fixed point $p_0=(1.584,0.431)^{T}$ of $\overline{M2}$ and, as we
can see on Figures \ref{osc12-1st}, \ref{zoomosc-1st} and Figures
\ref{osc12-2nd}, \ref{zoomosc-2nd}, the first pullback solution is a
good approximation of the periodic solution revolving around the
fixed point.

Finally, note that the notion of equivalence  considered in this
paper (Definition \ref{Equivalence}) is the appropriate  one from
the numerical perspective described here in order to compute
$T$-periodic solutions, in view of the fact that the main ingredient
was to achieve a correspondence between fixed points (with all the
periods) of an autonomous system and $T$-periodic orbits of a
nonautonomous one.

\section{Appendix 1: Solution of the time-dependent pullback problem
via Lie transforms}\label{appendix1} Let $p$ be an integer greater
than or equal to one. Let $A(\tau, x)$ be a smooth vector field on
$\mathbb{R} \times \mathbb{R}^{p}$ expanded in powers of $\tau $ as
\begin{equation}
A(\tau,x)=A_{\tau }(x)=\sum_{n\geq 0}\frac{\tau ^{n}}{n!}
A_{n}^{(0)}(x), \label{f3}
\end{equation}
and let $H(\tau,t,x)$ be a smooth vector field on $\mathbb{R}
\times\mathbb{R}^{+}\times \mathbb{R}^{p}$ expanded as,
\begin{equation}
H(\tau ,t,\xi )=H_{\tau }(t,\xi )=H_{\tau ,t}(\xi )=\sum_{n\geq
0}\frac{\tau ^{n}}{n!} H_{n}(\xi ,t)=\sum_{n\geq 0}\frac{\tau
^{n}}{n!}H_{n,t}(\xi ),  \label{f4}
\end{equation}
where $t$ is $\it{fixed}$. Denote by $\Xi_t$ the {\em semiflow}
generated by $H_t$, namely the semiflow of
\begin{equation}\label{systemeSH}
S_{H_{t}}\left\{
\begin{array}{l}
\frac{d\xi }{d\tau }=H(\tau ,t,\xi ) \\
\end{array}
\right. ;\, \xi \in \mathbb{R}^{p},\tau \in \mathbb{R}.
\end{equation}
Define the Lie transform {\em generated} by $H_t$ of $A$ evaluated
at $\tau$, denoted by $L(H_t)(\tau)\cdot A$, as the vector field:
\begin{equation}\label{Lietransf}
L(H_t)(\tau)\cdot A=A_{0}^{(0)}+\sum_{m\geq 1}\frac{\tau ^{m}
}{m!}A_{0,t}^{(m)},
\end{equation}
where the sequence of vector fields (acting on $\mathbb{R}^{p}$)
 $\{A_{0,t}^{(m)}\}$, is calculated from the sequence $\{A_{n}^{(0)}\}$ given by (\ref{f3}),
 using the recursive formula:
\begin{equation}  \label{f10}
A_{n,t}^{(i+1)}=A_{n+1,t}^{(i)}+\sum_{k=0}^{n}C_{n}^{k}L_{H_{n-k,t}}A_{k,t}^{(i)}
;\mbox{ for all }(i,n) \in \mathbb{Z}^{2}_{+},
\end{equation}
with $L_{H_{n-k,t}}$ expressing the Lie derivative with respect to
$H_{n-k,t}$ (see e.g. \cite{ch, v}, for details).

 A classical result is the following theorem which expresses the
pullback of a vector field as a Lie transform which can be proved
easily, by adapting, for instance, the work of \cite{her} relying on
Taylor series expansion and the key formula
\begin{equation}
\frac{d}{d\tau}(\Xi_{\tau,t})^{\ast}A_{\tau}=(\Xi_{\tau,t})^{\ast}(L_{H_{\tau,t}}A_{\tau}+\partial_{\tau}A_{\tau})
\end{equation}
from the realm of differential geometry (e.g., \cite{am, l}).

\vspace{1ex}

\noindent{\bf Theorem 7.1.}\label{reprepullback} {\em Let $\tau \in
\mathbb{R}$, as above. The pullback at time $t$ of $A_\tau$ by
$\Xi_{\tau,t}$ generated by $H_t$ is the Lie transform generated by
$H_t$ of $A$, evaluated at $\tau$, that is:
$$(\Xi_{\tau,t})^{\ast}A_{\tau}=L(H_{t})(\tau)\cdot A, \mbox{ for all
t}
 \in \mathbb{R}^{+}.$$}
%\refstepcounter{theorem}

Let $B$ be another smooth vector field on $\mathbb{R} \times
\mathbb{R}^{p}$ with formal expansion
\begin{equation}
B_{\tau }=\sum_{m\geq 0}\frac{\tau ^{m}}{m!}B_{m}.
\label{Bexpansion}
\end{equation}

One of the advantages of considering pullback as a Lie transform, is
that the framework of the latter yields linear conditions for
finding $\Xi_{\tau,t}$ which satisfies
$(\Xi_{\tau,t})^{\ast}A_{\tau}=B_{\tau}$. Indeed, using the formal
power series expansions (\ref{Lietransf}) and (\ref{Bexpansion}),
$(\Xi_{\tau,t})^{\ast}A_{\tau}=B_{\tau}$ leads to a sequence of
recursive linear PDEs (e.g. \cite{v}), namely,
\begin{equation}\label{ELm}
A_{0,t}^{(m)}=B_{m},\, m\in \mathbb{Z}_{+},\, t \in \mathbb{R}^{+},
\end{equation}
where the $H_{n,t}$ ($n \leq m$) present in (\ref{f4}) are the
unknowns contained in $A_{0,t}^{(m)}$, determining by this way the
generator $H_t$ of the Lie transform.

These equations are usually called Lie's equations.

As a consequence, we can state the following corollary of Theorem
7.1:

\vspace{1ex}

\noindent {\bf Corollary 7.2.}\label{existdiffeo} {\em A necessary
condition for the existence of a two-parameter family of
diffeomorphisms $(\Xi_{\tau,t})_{(\tau,t)\in
\mathbb{R}\times\mathbb{R}^{+}}$, generated by the one-parameter
family of vector fields $(H_t)_{t\in \mathbb{R}^{+}}$ given by
(\ref{f4}), such that $(\Xi _{\tau,t })^{\ast} A_{\tau}=B_{\tau }$
for all $t\geq 0$ and all $\tau\in\mathbb{R}$, is that Lie's
equations (\ref{ELm}) are solvable for the unknowns $H_{m,t}$ for
all $(m,t)\in\mathbb{Z}_{+}\times\mathbb{R}^{+}$.}

\section{Appendix 2 (added for the ArXiv version): Proof of  Lemma 6.1 of Chekroun {\it et al.}, DCDS, {\bf 14}(4), 2006.}
We give here the proof of the Lemma 6.1 of Chekroun {\it et al.},
DCDS, {\bf 14}(4), 2006, that corresponds here also to Lemma 6.1 of
the present manuscript.

\begin{proof}{\em of Lemma 6.1.} As the vector field are assumed to be
complete on $\Omega$, there exists an open subset $V\subset \Omega$,
such that the time $T$-map associated with $Z$, denoted by $P$ is
well-defined on $V$. We denote by $\widetilde{P}$ the time $T$-map
associated with the periodic vector field $Y$.

First of all, we suppose that $Y$ and $Z$ are
$\mathcal{P}^{r}_{diff}$-equivalent, so there exists a map
$\Phi\in\mathcal{P}_d^{k}(\Omega)$, such that:
\begin{equation}\label{ConjugSol}
x(t,x_0)=\Phi_t(z(t,\Phi_0^{-1}(x_0))), \mbox{ for all }, x_0\in
\Omega,
\end{equation}
using the notations of Definition \ref{Equivalence}.

Since for all $\xi \in V$,  $\widetilde{P}(\xi)=x(T,\xi)$ and
$P(\xi)=z(T,\xi)$, then we get from (\ref{ConjugSol}):
$$\widetilde{P}\circ \Phi_0(x_0)=x(T, \Phi_0(x_0))=\Phi_T(z(T,x_0))=\Phi_0(z(T,x_0))=\Phi_0\circ P(x_0),$$
and the necessary condition of the lemma is satisfied.

For the sufficient condition, let us denote by $z_t:z_0\rightarrow
z(t,z_0)$ the flow of $Z$ such that $z(0,z_0)=z_0$, and by
$x_t:x_0\rightarrow x(t,x_0)$, the semiflow of $Y$ such that
$x(0,x_0)=x_0$. Then $z_{t+T}=z_t\circ P$ (at least in $V$) and
$x_{t+T}=x_t\circ \widetilde{P}$. By assumption there exists a
$\mathcal{C}^{r}$-diffeomorphism $H$ such that $\widetilde{P}\circ H
=H\circ P$. Let us introduce for all $t \geq 0 $,
$$H_t=x_t\circ H\circ (z_t)^{-1}.$$
Then $H_t$ is a time-dependent $\mathcal{C}^{r}$-diffeomorphism with
$H_0=H$. Obviously $H_t$ carries solution $z(t,z_0)$ into
$x(t,H^{-1}(z_0))$; and we have:
\begin{equation}
H_{t+T}=x_{t+T}\circ H \circ (z_{t+T})^{-1}=x_t\circ
\widetilde{P}\circ H\circ P^{-1}\circ (z_t)^{-1} =x_t \circ H \circ
(z_t)^{-1}=H_t,
\end{equation}
that gives the $T$-periodicity of the change of variables.
\end{proof}

\section*{Acknowledgements} The authors would like to thank the anonymous referee for his insightful
comments and valuable suggestions. The research of M. Ghil and F.
Varadi was supported in part by the U.S. National Science Foundation
under a grant from the Divisions of Atmospheric Sciences and of
Mathematical Sciences.

\vspace{2ex}
%\medskip
%Corresponding author: \verb"chekro@lmd.ens.fr"
\medskip


\begin{thebibliography}{99}


\bibitem{am} R. Abraham and J.E. Marsden, ``Foundation of
Mechanics", 2nd ed., Benjamin Cummings Publishing Co., Inc.,
Reading, MA., 1978.

\bibitem{a} V.I. Arnold, Geometrical Methods in the Theory of
Ordinary Differential Equations, 2nd ed., Springer-Verlag, 1988.

\bibitem{b} A. Blaga, {\em Perturbation methods with Lie
series}, Stud. Univ. Babes-Bolyai Mathem. XL, {\bf 3} (1995),
11--28.

\bibitem{bl} A. Buic\u{a} and J. Llibre, {\em Averaging methods for finding
periodic orbits via Brouwer degree}, Bull. Sci. Math., {\bf 128}
(2004), 7--22.

\bibitem{ch} S.N. Chow and J.K. Hale, Methods and Theory of
Bifurcations, Springer Verlag, 1982.

\bibitem{ra} V.T. Coppola and R.H. Rand, {\em Computer
algebra implementation of Lie transforms for Hamiltonian systems:
Application to the Nonlinear Stability of $L_4$}, ZAMM. Z. Angew.
Math. Mech., {\bf 69} (1989), no.9, 275--284.

\bibitem{d} A. Deprit, {\em Canonical transformations depending on a
small parameter}, Celestial Mech., {\bf 1} (1969), 12--30.

\bibitem{det} M. D. Dettinger, Variation of Continental Climate
and Hydrology on Diurnal-to-Interdecadal Times Scales, Ph. D.
Thesis, University of California, Los Angeles, 1997.

\bibitem{dra} J. Dragt, {\em Lie methods for nonlinear
dynamics with applications to accelerator physics}, University of
Maryland Physics Department Report (2000).

\bibitem{esd} J.A. Ellison, A.W. Saenz, and H.S. Dumas, {\em Improved Nth
order averaging theory for periodic systems}, J. Differential
Equations, {\bf 84} (1990), 383--403.

\bibitem{fra} L.E. Fraenkel, {\em Formulae for high derivatives of
composite functions}, Math. Proc. Cambridge Philos. Soc., {\bf 83}
(1978), 159--165.

\bibitem {gab} I. Gabitov, T. Schafer and S.K. Turitsyn, {\em
Lie-transform averaging in nonlinear optical transmission systems
with strong and rapid periodic dispersion variations}, Phys. Lett.
A, {\bf 265} (2000), 274--281.

\bibitem{gh} J. Guckenheimer and P. Holmes, Nonlinear
Oscillations, Dynamic Systems, and Bifurcations of Vector Fields,
Springer-Verlag, 1983.

\bibitem{hale} J.K. Hale, Ordinary Differential Equations, 2nd ed.,
Robert E. Krieger Publishing Company, Inc. malabar, Florida, 1980.

\bibitem{har} Hartano and A.H.P. van der Burgh, {\em Higher-order
averaging: periodic solutions, linear systems and an application},
Nonlinear Anal., {\bf 52} (2003), 1727--1744.

\bibitem{he} J. Henrard, {\em On a perturbation theory using Lie
transforms}, Celestial Mech., {\bf 3} (1970), 107--120.

\bibitem{hen} J. Henrard, {\em The algorithm of the inverse for Lie
transform}, Recent Advances in Dynamical Astronomy, D. Reidel
Publishing Company, (1973), 250--259.

\bibitem{hen2} J. Henrard, {\em The adiabatic invariant in classical mechanics}, in ``Dynamics
Reported, 2 new series"(eds. C.K.R.T. Jones, U. Kirchgraber and H.O.
Walther), Springer-Verlag, (1993), 117--235.

\bibitem{her} J. Henrard and J. Roels, {\em Equivalence for Lie
transforms}, Celestial Mech., {\bf 10} (1974), 497--512.

\bibitem{hem} P.G. Hess and S. Madronich, {\em On tropospheric
chemical oscillations}, J. Geophys. Res., {\bf 102} (1997),
15,949--15,965.

\bibitem{h} G. Hori, {\em Theory of general perturbations with
unspecified canonical variables}, Publication of the Astronomical
Society of Japan, {\bf 18} (1966), no. 4, 287--296.

\bibitem{hu} J. Hubbard and Y. Ilyashenko, {\em A proof of
Kolmogorov's theorem}, Discrete Contin. Dyn. Syst., {\bf 10}, no.
1,2, (2004), 367--385.

\bibitem{k1}  A. Kamel, {\em Perturbations method in the theory of
nonlinear oscillation}, Celestial Mech., {\bf 3} (1970), 90--106.

\bibitem{k2}  A. Kamel, {\em Lie transforms and the
Hamiltonization of non-Hamiltonian systems}, Celestial Mech., {\bf
4} (1971), 397--405.

\bibitem{ko2} P.V. Koseleff, Calcul Formel pour les M\'{e}thodes
de Lie en M\'{e}canique Hamiltonienne, Ph. D. Thesis, Ecole
Polytechnique, 1993.

\bibitem{ko} P.V. Koseleff, {\em Comparison between Deprit and Dragt-Finn
perturbation methods}, Celestial Mech. Dynam. Astronom., {\bf 58}
(1994), no. 1, 17--36.

\bibitem{kra} M.A. Krasnosel'skii and P.P. Zabrei\u{i}ko,
Geometrical Methods of Nonlinear Analysis, Springer-Verlag, 1984.

\bibitem{l} S. Lang, Differential Manifolds, Addison Wesley,
1972.

\bibitem{lo} P. Lochak and C. Meunier, Multiphase
Averaging for Classical Systems, Applied Mathematical Sciences, 72,
Springer-Verlag, 1988.

\bibitem{mur} J.A. Murdock, {\em Qualitative theory of nonlinear
resonance by averaging and dynamical systems methods}, in ``Dynamics
Reported, 1"(eds. U. Kirchgraber and H.O. Walther), Springer-Verlag,
(1988), 91--172.

\bibitem{na} A.H. Nayfeh, Perturbation Methods,
Wiley-Interscience Publication, John Wiley \& Sons, 1973.

\bibitem{per} L.M. Perko, {\em Higher order averaging and related
methods for perturbed periodic and quasi-periodic systems}, SIAM J.
Appl. Math., {\bf 17} (1968), no. 4, 698--724.

\bibitem{sv} J.A. Sanders and F. Verhulst, Averaging Methods in
Nonlinear Dynamical Systems, Springer-Verlag, 1985.

\bibitem{ste} D. Steichen, {\em An averaging method to
study the motion of lunar artificial satellites I, II}, Celestial
Mech. Dynam. Astronom., {\bf 68} (1998), 205--247.

\bibitem{v} F. Varadi, {\em Branching solutions and Lie
series}, Celestial Mech. Dynam. Astronom., {\bf 57} (1993),
517--536.

\bibitem{jap} K. Yagasaki and T. Ichikawa, {\em Higher-order averaging for
periodically forced, weakly nonlinear systems}, Int. J. Bifurc.
Chaos, {\bf 9} (1999), no. 3, 519--531.

\end{thebibliography}
\end{document}